\documentclass[aihp]{imsart}

\RequirePackage{amsthm,amsmath,amsfonts,amssymb}
\RequirePackage[numbers,sort&compress]{natbib}
\RequirePackage[colorlinks,citecolor=blue,urlcolor=blue]{hyperref}
\RequirePackage{graphicx}

\startlocaldefs
\theoremstyle{plain}
\newtheorem{axiom}{Axiom}
\newtheorem{claim}[axiom]{Claim}
\newtheorem{theorem}{Theorem}[section]
\newtheorem{lem}[theorem]{Lemma}
\newtheorem{prop}[theorem]{Proposition}
\newtheorem{coro}[theorem]{Corollary}
\theoremstyle{remark}
\newtheorem{defn}[theorem]{Definition}

\newtheorem{rem}{Remark}
\newtheorem*{fact}{Fact}
\def\RR{\mathbb{R}}
\def\ZZ{\mathbb{Z}}
\def\NN{\mathbb{N}}
\def\S{\mathcal{S}}
\def\V{\mathcal{V}}
\def\L{\mathcal{L}}
\def\T{\mathcal{T}}
\def\one{{\mathbf 1\!\!1}}
\def\fora{\forall\,}

\endlocaldefs

\begin{document}

\begin{frontmatter}
\title{Stochastic ordering, attractiveness and couplings in non-conservative particle systems}
\runtitle{Stochastic ordering of non-conservative particle systems}

\begin{aug}
\author[A]{\inits{R.}\fnms{Raúl}~\snm{Gouet}\ead[label=e1]{rgouet@dim.uchile.cl}\orcid{0000-0003-3125-6773}},
\author[B]{\inits{F.J.}\fnms{F. Javier}~\snm{López}\ead[label=e2]{javier.lopez@unizar.es}\orcid{0000-0002-7615-2559}}
\and
\author[B]{\inits{G.}\fnms{Gerardo}~\snm{Sanz}\ead[label=e3]{gerardo.sanz@unizar.es}\orcid{0000-0002-6474-2252}}
\address[A]{Depto. de Ingenier\'{\i}a Matem\'{a}tica and CMM (CNRS, IRL 2807), Universidad de Chile. Santiago, Chile\printead[presep={,\ }]{e1}}

\address[B]{Dpto. M\'{e}todos Estad\'\i sticos and BIFI, Facultad de Ciencias, Universidad de Zaragoza. Zaragoza, Spain\printead[presep={,\ }]{e2,e3}}
\end{aug}

\begin{abstract}
We analyse the stochastic comparison  of interacting particle systems allowing for multiple arrivals, departures and non-conservative jumps of individuals between sites. That is, if $k$ individuals leave site $x$ for site $y$, a possibly different number $l$  arrive at destination. This setting includes  new models, when compared to the conservative case, such as metapopulation models with deaths during migrations. It implies a sharp increase of technical complexity, given the numerous changes to consider. Known results are significantly generalised, even in the conservative case,  as no particular form of the transition rates is assumed.

We obtain necessary and sufficient conditions on the rates for the stochastic comparison of the processes and  prove their  equivalence with the existence of an order-preserving Markovian coupling. As a corollary, we get necessary and sufficient conditions for the attractiveness of the processes. A salient feature of our approach lies in the presentation of the coupling in terms of solutions to network flow problems. 

We illustrate the applicability of our results to a flexible family of population models described as interacting particle systems, with a range of parameters controlling  births, deaths, catastrophes or migrations. We provide explicit conditions on the parameters for the stochastic comparison and attractiveness of the models, showing their usefulness in studying their limit behaviour. Additionally, we give three examples of constructing the coupling. 

\end{abstract}

\begin{abstract}[language=french]
Nous étudions la comparaison stochastique des systèmes de particules en interaction permettant de multiples arrivées, départs et sauts non conservatifs d'individus entre les sites. Autrement dit, si $k$ individus quittent le site $x$ pour le site $y$, un nombre éventuellement différent $l$ arrive à destination. Ce cadre inclut de nouveaux modèles, par rapport au cas conservatif, tels que les modèles de métapopulation avec des décès pendant les migrations. Cela implique une augmentation nette de la complexité technique, compte tenu des nombreuses modifications à prendre en compte. Des résultats connus sont considérablement généralisés, même dans le cas conservatif, car aucune forme particulière des taux de transition n'est supposée.

Nous obtenons des conditions nécessaires et suffisantes sur les taux pour la comparaison stochastique des processus et prouvons leur équivalence avec l'existence d'un couplage markovien préservant l'ordre. En corollaire, nous obtenons des conditions nécessaires et suffisantes pour l'attractivité des processus. Une caractéristique saillante de notre approche réside dans la présentation du couplage en termes de solutions à des problèmes de flux de réseau.

Nous illustrons l'applicabilité de nos résultats à une famille flexible de modèles de population décrits comme des systèmes de particules en interaction, avec une gamme de paramètres contrôlant les naissances, les décès, les catastrophes ou les migrations. Nous fournissons des conditions explicites sur les paramètres pour la comparaison stochastique et l'attractivité des modèles, démontrant leur utilité dans l'étude de leur comportement limite. De plus, nous donnons trois exemples de construction du couplage.
\end{abstract}



\end{frontmatter}

\section{Introduction}

\label{sec:intro}
Stochastic comparison is an essential tool in the study of interacting particle systems (IPS).  See Chapter 5 of \cite{MS02} for basic definitions and general results on stochastic ordering between random processes. Among many orders between random variables or processes, the stochastic order defined through the coordinate-wise ordering of configurations is widely used in IPS, as it helps to determine the set of invariant measures of the processes; see \cite{Lig} and references therein. 
  Since IPS are Markov processes, usually defined in terms of their transition rates, conditions on the rates are needed for  their stochastic comparison. Ideally these conditions should be easy to check and  write via inequalities including only each site and its neighbours. The stochastic comparison of continuous-time Markov chains (i.e., for countable spaces) is well understood.  Necessary and sufficient conditions on the rates for comparability are given in \cite{Mas}. These conditions have been applied to different models, including birth-death-migration processes on finite graphs; in particular, sufficient conditions for comparability of a fairly general class of these processes are found in \cite{DS95}, while necessary and sufficient conditions are given in \cite{DLS14}, when migrations only involve two nodes at a time. In the case of Markov processes on uncountable state spaces, a general comparison result, as seen in the countable case, does not exist but results similar to those  in \cite{Mas} are known in a variety of settings. For instance,  sufficient conditions are given in \cite{BL94} for the comparability of jump processes (not necessarily Markovian) on partially ordered Polish spaces.  Also, necessary and sufficient conditions for diffusions on $\RR^n$ are presented in \cite{CW93}. The basic result on comparability of IPS is Theorem III.1.5 of \cite{Lig}, for spin systems (i.e. 0-1 IPS where only one site changes its value at each transition). This was  extended in \cite{FF97} to the case in which the state space of each site is a finite set endowed with a partial order. For IPS where more than one site can change at each transition, the situation is more complicated. Necessary and sufficient conditions for 0-1 IPS, built up from a spin system and an exclusion process, are found in \cite{FF97}. Sufficient conditions are given in \cite{DLS04} for the stochastic comparison of a class of IPS involving changes at two sites, with an arbitrary partial order and, under additional restrictions, these conditions are also shown to be necessary. In the present work we assume that the set of values at each site has a total order, which is by far the most common situation  in the literature.  Closely related papers devoted to the stochastic comparison of infinite volume migration processes are \cite{Bor} and \cite{GS}. In \cite{GS} 
necessary and sufficient conditions are given for the attractiveness of processes with jumps of  $k\ge1$ individuals. The inclusion of arrivals and departures  in that model is analysed in \cite{Bor}, where sufficient conditions for comparability are obtained; moreover, when both processes have equal matrices defining their jump rates, the conditions are also necessary.   In particular, the conditions are necessary and sufficient for attractiveness, since this involves the comparison of a process with itself.  In \cite{Bor} and \cite{GS} the  proof is based on the construction of an explicit order-preserving Markovian coupling (OMC for short). In   \cite{Bor} this construction is rather difficult and lengthy, requiring a detailed analysis of the transitions of the processes; see pages 121--141 in that paper.

 A problem related to  the comparability of two processes is the existence of an OMC between them. It is known (Theorem 5 in \cite{KKO}) that the comparability of two Markov processes, on a partially ordered Polish space, is equivalent to the existence of an order-preserving coupling, but the question is if such coupling can be chosen to be a Markov process. 
 The answer to this question is positive in several instances, such as Markov chains on countable state spaces \cite{LMS}, while it is negative in the case of multidimensional diffusions \cite{WX97}. In the case of IPS, the answer is positive for 0-1 spin systems, via the so-called basic coupling (Section III.1 in \cite{Lig}). For more general IPS the problem was addressed in \cite{FF97} but an answer was only found under restrictive conditions on the rates. For IPS where only one particle can change at each transition, the question was answered (in the positive) in \cite{LS}. More recently, a positive answer has also been provided for general exclusion processes in \cite{GS23}. General theory and applications of Markovian couplings can be found in Chapter 2 of \cite{chen2006}.

Our goal in this paper is to study the comparability of more general IPS  than those considered in \cite{Bor} and \cite{GS}. We work with processes allowing arrivals, departures and jumps of individuals but  jumps need not be conservative, that is, a batch of $k$ individuals leaving site $x$ arrives at site $y$ as a batch of $l$ individuals, where $k,l\ge1$. Allowing $k\ne l$ can be interpreted as the batch being augmented by some individuals entering the system ($k<l$) or reduced by some individuals leaving the system $(k>l)$. The latter situation includes the possibility of deaths during migrations, which has been considered in the literature of population dynamics; see \cite{Arr03}, \cite{Gri01}, \cite{HAM}  and \cite{Zho04}. Note that IPS having jumps with $k\ne l$ cannot be analysed using the results in \cite{DLS04}, since condition (2) in that paper is not satisfied, except for some particular cases such as  when births, deaths and migrations are restricted to a single individual or one of the processes does not allow migration (Sections 5.1, 5.2 in that paper). 

In our main result we establish necessary and sufficient conditions on the rates for the stochastic comparability of the processes and prove its equivalence with the existence of an OMC. As a corollary, we give  necessary and sufficient conditions for attractiveness.  Moreover, we provide an algorithm for the construction of the OMC, whose rates can be found as the solutions to network flow problems.  Network flow theory  has been applied to the construction of couplings of Markov processes in \cite{DLS04} and \cite{LS}.  Having an OMC is useful because it can be a powerful tool in the study of the long term behaviour of IPS; see, e.g. \cite{Lig13} in the case of spin systems;  \cite{Sto20}, \cite{Tzi19}, for processes where only one site changes its value at each transition and \cite{BMRS}, \cite{FAJ2016}, \cite{GS}, \cite{GS23}, \cite{SS18}, for migration processes. An interesting feature of our approach is the possibility of considering an objective function in the network flow problems so that the coupling has, besides the preservation of the order,  properties potentially useful in the study of the processes (see Remark \ref{objetivo}).

Our results significantly extend those of \cite{Bor} and \cite{GS}. First we allow migrations to be non-conservative.  This includes new  models (see Section \ref{secex}) and implies a higher complexity in the construction of the  OMC, since there are many more changes to consider. In particular, it seems that a construction of the coupling following the ideas in \cite{Bor} is not possible. Second,  the models in \cite{Bor} and \cite{GS} assume that the jump rates of particles from site $x$ to site $y$ depend on the number of individuals at $x$ and $y$, but not on the number of individuals in the rest of sites. While this framework is sufficiently general to include many interesting models, such as the generalized misanthrope process, it does not encompass others, such as threshold models, for example. We do not make such  an assumption.
Third, the conditions in \cite{Bor} are necessary and sufficient for stochastic comparability when the matrices $p$ appearing in the rates of the two processes are equal, but only sufficient conditions are given when they are different (see Theorem 2.4 and Corollary 3.28 in \cite{Bor}). Our conditions \eqref{eq:c1} and \eqref{eq:c2} are necessary and sufficient for stochastic comparison,  without any further constraints on the rate structures. 

The paper is organized as follows. We begin with the definition of the processes and notation used throughout the paper. In Theorem \ref{teorema} of Section \ref{sec:stochorder} we state the main result of the paper, specifically, the equivalence between conditions \eqref{eq:c1} and \eqref{eq:c2} with the stochastic comparability 
 of the processes and the existence of an OMC. The initial steps of the proof are   presented in this section as well. Also, Corollary \ref{corolario_princiipal} gives necessary and sufficient conditions for attractiveness.  Section \ref{tnfp} is dedicated to stating and solving  several network flow problems, used in the construction of an OMC. The construction itself, along with the remaining steps of the proof of Theorem \ref{teorema}, is detailed in Section \ref{sec:coc}. 
  In Section \ref{secex} we apply our results to an extended version of the models considered in \cite{Bor2012}, including non-conservative migrations, with rates depending on the neighbours. For these models we give conditions for stochastic comparison and attractiveness, showing their applicability in the study of their limit behaviour. We also illustrate the construction of the coupling in three particular instances.
 
\subsection{Notation and preliminaries}
\label{sec:notation}
This paper focuses on interacting particle systems (IPS) $(\zeta_t)_{t\ge0}$, ($(\zeta_t)$ for short), with state space $\Omega=W^\S$, where $\S\subseteq\ZZ^d$ and $W\subseteq\NN\cup\{0\}$. As usual, $\ZZ, \NN$ denote the sets of integers and positive integers, respectively.

Transitions are described through local maps $\sigma_x^{\pm k}, \sigma_{xy}^{\pm k,l}$ on $\Omega$ (also referred to as ``changes''), defined by
\begin{equation*}
(\sigma_x^{+k}\zeta)(z)= \begin{cases}\zeta(x)+ k &\mbox{if }z=x,\\
\zeta(z) & \mbox{if } z\ne x,
\end{cases}\qquad\qquad\qquad (\sigma_x^{-k}\zeta)(z)= \begin{cases}\zeta(x)-k &\mbox{if }z=x,\\
\zeta(z) & \mbox{if } z\ne x,
\end{cases}
\end{equation*}
\begin{equation*}
(\sigma_{xy}^{+ k,l}\zeta)(z)= \begin{cases}\zeta(x)+ k &\mbox{if }z=x, \\ \zeta(y)- l&\mbox{if } z=y,\\
\zeta(z) & \mbox{if }  z\not\in \{x,y\},
\end{cases}\qquad\quad\, (\sigma_{xy}^{- k,l}\zeta)(z)= \begin{cases}\zeta(x)- k &\mbox{if }z=x, \\ \zeta(y)+ l&\mbox{if } z=y,\\
\zeta(z) & \mbox{if }  z\not\in \{x,y\},
\end{cases}
\end{equation*}
where  $x,y\in \S, x\ne y, k,l\in\NN$ and $\zeta$ is in the domain of the corresponding map. The  domain of map $a$ is defined as $\Omega_a=\{\zeta\in\Omega: a\zeta\in\Omega\}$. The set of all changes (maps) involving site $x$ and possibly another site, is denoted by $C^x$ and the set of all changes is denoted by $C$. That is, 
\begin{equation*}
	C^x=\{\sigma_x^{+ k}, \sigma_x^{- k}, \sigma_{xy}^{+k,l}, \sigma_{xy}^{-k,l}: k,l\in \NN, y\in \S\}\quad \text{and}\quad C=\bigcup_{x\in \S}C^x.
\end{equation*}
For simplicity, given $a\in C, \zeta\in\Omega_a$, we write $\zeta_a$ instead of $a\zeta$ and denote by $c_a(\zeta)$ its rate. With this notation, the infinitesimal generator $\cal L$ of $(\zeta_t)$, acting on a function $h:\Omega\to\RR$, can be  written as 
\begin{equation}
	\label{eq:generatoreta}
	\L h(\zeta)=\!\!\sum_{a\in C:\zeta\in \Omega_a}\!\!c_a(\zeta)(h(\zeta_a)-h(\zeta)), \quad\zeta\in\Omega.
\end{equation}
Also, if $\zeta\not\in\Omega_a$,  we may set $c_a(\zeta)=0$ and define $\zeta_a$ arbitrarily so that the generator in \eqref{eq:generatoreta}, and similar expressions, become slightly simpler because no explicit mention of $\Omega_a$ is needed. 

\subsection{Existence of the processes}
\label{sec:existence}
General  results about the existence of IPS when the state space $\Omega$ is finite or compact can be found, for instance, in \cite{Lig}. However, since in our context $W$ is possibly infinite, $\Omega$ may not be compact and such results are not applicable. Instead, we can invoke the theory established in \cite{Pen}, which depends  on the way sites interact.

For that purpose we assume that $\S$ is endowed with a distance $d_\S$, related, for example, to the $L_1$ norm, which is used to define neighbourhoods as follows: For a given $\delta>0$ and for all $x,y\in\S$, let $\V(x)=\{z\in \S: d_\S(x,z)\le\delta\}$ and write $y\sim x$ if $y\ne x$ and $y\in \V(x)$. Of course, $x\sim y$ and $y\sim x$ are equivalent. Only migrations between $x\sim y$ are allowed in the processes, that is, $c_a(\zeta)=0, \fora \zeta\in \Omega, a=\sigma_{xy}^{\pm k,l}$, such that $y\not\in\V(x)$.
Also, let $\mathcal{W}(x)$ be the set of sites $z\in\mathcal{S}$ such that the rate of change of $x$ depends on $z$, that is, $c_a(\zeta)$ does not depend on $\zeta(z)$, $\fora a\in C^x, z\not\in\mathcal{W}(x), \zeta\in\Omega$. We assume that the cardinality of $\mathcal{W}(x)$ is uniformly bounded which, along with the definition of $\V(x)$, implies condition (2.2) in \cite{Pen}.

We also make the following boundedness assumption
\begin{equation}\label{condexist2}
	\sup\left\{\sum_{a\in C^x}c_a(\zeta): x\in\S, \zeta\in\Omega\right\}<\infty,
\end{equation}
which clearly entails (2.4) in \cite{Pen}. Therefore, under the assumptions discussed above,  Theorem 2.1 of that paper guarantees that \eqref{eq:generatoreta} is a generator that defines a unique Markov process.
\subsection{Stochastic ordering}
The natural order on $W$ induces the coordinate-wise (or site-wise) order on $\Omega$, defined by  
\begin{equation*}
	\eta\le\xi\quad\text{ if }\quad\eta(x)\le\xi(x),\,\fora x\in \S,
\end{equation*}
 which in turn induces an order on the set of probabilities on $\Omega$, given by 
\begin{equation*}
	 \mu\le\nu\quad \text{if} \quad\int fd\mu\le\int fd\nu,\quad \fora f:\Omega\to\mathbb{R}\text{ increasing}.
\end{equation*} 
Stochastic domination between two processes $(\eta_t)$ and $(\xi_t)$ with respective semigroups $\T_1,\T_2$, denoted  $(\eta_t)\le_{st}(\xi_t)$, is defined as follows: 
\begin{equation*}(\eta_t)\le_{st}(\xi_t) \quad \text{if} \quad \mu\le\nu \ \text{implies} \ \mu \T_1(t)\le \nu \T_2(t), \ \fora\, t\ge0.\end{equation*} A process 
$(\eta_t)$ is attractive if $(\eta_t)\le_{st}(\eta_t)$.

Further, an OMC of $(\eta_t)$ and $(\xi_t)$ is a bivariate Markov process $((\eta'_t,\xi'_t))$ on $\Omega\times \Omega$, with marginals equally distributed as $(\eta_t)$ and $(\xi_t)$, such that $P_{\eta,\xi}[\eta'_t\le\xi'_t]=1, \fora\eta\le\xi, t\ge0$, where $P_{\eta,\xi}$ is the distribution of $(\eta'_t,\xi'_t)$, starting at $(\eta,\xi)$.

\section{Main result} 

\label{sec:stochorder}
For $x\in \S$ and $\eta,\xi\in\Omega$ such that $\eta\le\xi$, let
\begin{equation*}
	R_1^{x}=\{a\in C^{x}: \eta_a(x)>\xi(x)\}, \quad	R_2^{x}=\{b\in C^{x}: \eta(x)>\xi_b(x)\}.
\end{equation*}
Also,
for any $D_1\subseteq R_1^{x}$, $D_2\subseteq R_2^{x}$, let
\begin{equation*}
	D_1^\uparrow=\{b\in C^{x}:\ \exists a\in D_1\text{ s.t. }\eta_a\le \xi_b\},\quad D_2^\downarrow=\{a\in C^{x}:\ \exists b\in D_2\text{ s.t. }\eta_a\le \xi_b\}.
\end{equation*}
\begin{rem}
	\label{rem:changes}
In the definitions above, it is important to bear in mind that $R_1^{x}, R_2^{x}$ depend on $\eta$ and $\xi$ but this dependence is not shown, for the sake of simplicity. Also, $R_1^{x}$ ($R_2^{x}$)  and $D_2^\downarrow$ ($D_1^\uparrow$) only contain changes $a$ such that $\eta\ (\xi)\in \Omega_a$.
\end{rem}
Suppose that  $(\eta_t), (\xi_t)$ are two IPS with respective semigroups $\T_i$ and generators 
\begin{equation}
	\label{eq:genLi}
	\L_i h(\zeta)=\sum_{a\in C}c_a^i(\zeta)(h(\zeta_a)-h(\zeta)),\  i=1,2.
\end{equation}
and consider the following conditions on their rates:
\begin{alignat}{1}
	\label{eq:c1}
	\sum_{a\in D_1}c_a^1(\eta)&\le\sum_{b\in D_1^\uparrow}c^2_b(\xi), \quad\fora\ D_1\subseteq R_1^{x},\\
	\label{eq:c2}
	\sum_{b\in D_2}c^2_b(\xi)&\le\sum_{a\in D_2^\downarrow}c_a^1(\eta), \quad\fora\ D_2\subseteq R_2^{x}.
\end{alignat}

Expressions \eqref{eq:c1} and \eqref{eq:c2} are analogous to (5) and (6) in \cite{DLS14},  for the comparability of multicomponent systems on a countable state space. In that article they are shown to be necessary and sufficient because they are equivalent to the classic conditions for continuous-time Markov chains; see \cite{LMS} and \cite{Mas}.

The main result of the paper is the following.
\begin{theorem}
	\label{teorema} 
	Let $(\eta_t)$, $(\xi_t)$ be two IPS with respective semigroups $\T_1, \T_2$ and generators $\L_1, \L_2$, given in \eqref{eq:genLi}.
	The following statements are equivalent:
	\begin{enumerate}\item[(a)]  $(\eta_t)\le_{st}(\xi_t)$.\\
	\item[(b)]  Conditions \eqref{eq:c1} and \eqref{eq:c2} hold $\fora\, x\in \S, \eta,\xi\in\Omega $  \textup{s.t.}  $\eta\le\xi$.\\
	\item[(c)] There exists an OMC between $(\eta_t)$ and $(\xi_t)$.\end{enumerate}
	\end{theorem}
	\begin{proof} 
			Since clearly (c) implies (a), it remains to prove that (a) implies (b) and (b) implies (c). We show here that (a) implies (b) and leave the (more difficult) proof of the second implication for Sections \ref{tnfp} and \ref{sec:coc}, where an OMC is presented in Definition \ref{def:gencoup} and validated in Proposition \ref{prop:gencoup}. 
	
	Suppose that $(\eta_t)\le_{st}(\xi_t)$ and let $\eta\le\xi$. By the definition of generator,  
	$(\T_i(t)h(\zeta)-h(\zeta))/t\to\L_ih(\zeta)$, as $t\to0$, for all $\zeta\in\Omega,\ i=1,2,$ and any bounded
	function $h:\Omega\to\mathbb{R}$, depending on a finite number of sites. Moreover, if  $h$ is increasing and such that $h(\eta)=h(\xi)$, we get 
	\begin{equation*}
		\label{eq:L1leL2}
	{\cal L}_1h(\eta)\le{\cal L}_2h(\xi).
	\end{equation*}
	For $x\in \S$ and $D_1\subseteq R_1^{x}$, let  $h(\zeta)=\sup\limits_{a\in D_1}\one_{\{\eta_a\le\zeta\}}.$ Observe that $h$ is increasing and that $h(\eta)=h(\xi)=0$, by the definition of $R_1^{x}$. So,
	\begin{equation}
		\label{eq:L1}
		\sum_{a\in D_1}c^1_a(\eta)=\sum_{a\in D_1}c_a^1(\eta)h(\eta_a)\le \sum_{a\in C}c_a^1(\eta)h(\eta_a)={\cal L}_1h(\eta).
	\end{equation}		
	 We claim that $c_b^2(\xi)h(\xi_b)=0$,  $\fora b\in C\setminus C^{x}$. Indeed, if $b\in C\setminus C^{x}$ then either  $c_b^2(\xi)=0$  and the claim holds, or $c_b^2(\xi)>0$.  In the latter case, since site $x$ is not affected by $b$,  we have $\xi_b(x)=\xi(x)<\eta_a(x)$,  $\fora a\in D_1$ (by the definition of $R_1^{x}$), which yields $h(\xi_b)=0$ and the claim is verified. Further, if $b\in C^{x}\setminus D_1^\uparrow$, then there is no $a\in D_1$ such that $\eta_a\le \xi_b$, and so $h(\xi_b)=0$ as well. Also, $h(\xi_b)=1$,  $\fora b\in D_1^\uparrow$. Therefore, since $h(\xi)=0$, 
\begin{equation}
	\label{eq:L2}
	{\cal L}_2h(\xi)=\sum_{b\in  C}c_b^2(\xi)h(\xi_b)=\sum_{b\in D_1^\uparrow}c^2_b(\xi),
\end{equation} 
and  \eqref{eq:c1} follows from \eqref{eq:L1}, \eqref{eq:L2} and the inequality ${\cal L}_1h(\eta)\le {\cal L}_2h(\xi)$. The proof of \eqref{eq:c2} is analogous and is omitted.
\end{proof}

As a direct consequence of Theorem \ref{teorema} we obtain the following result on attractiveness.
\begin{coro}
	\label{corolario_princiipal} 
	Let $(\eta_t)$ be an IPS with generator \eqref{eq:generatoreta}. Then $(\eta_t)$ is attractive if and only if conditions \eqref{eq:c1} and \eqref{eq:c2} hold with $c^1_a(\eta)=c_a(\eta), c^2_a(\xi)=c_a(\xi)$,  $\fora\,  x\in \S, \eta,\xi\in\Omega$ such that $\eta\le\xi$.
\end{coro}

 \begin{rem}\label{condicionesgenerales}  In order  to fulfil  (2.2) and (2.4) of \cite{Pen}, in Section \ref{sec:existence} we state conditions which are sufficient (though not necessary) for the existence and definition of the processes in terms of their generators. 
We highlight that, for the construction of the coupling through the definition of the rates (see Section \ref{tnfp}), we only need  
\begin{equation*}\label{sumatotaltasas}\sum_{a\in C^{x}}(c_a^1(\eta)+c_a^2(\xi))<\infty,\quad \fora x\in S,\eta,\xi\in\Omega,
\end{equation*}
because this guarantees that the sums of the upper bounds in the flow problems of Section \ref{tnfp} are finite. For instance, when the set $W$ is finite, we can allow migrations between a site $x$ and and any other $y\in\S$, provided that conditions, such as (3.3) and (3.8), in I.3 of \cite{Lig}, are fulfilled.
 Of course, some additional conditions may be required to ensure that the rates define an IPS on $(W\times W)^{\S}$.
\end{rem}
\section{Network flow}\label{tnfp}

Network flow theory has proven valuable in constructing OMC in prior works.
 In \cite{LS}, where only a site can change at each transition, a network flow problem is defined for every $x\in \S$. Similarly, in \cite{DLS04}, where two sites can change together, a problem is also defined for each site but a condition on the rates (condition (2) in that paper) ensures that problems for different sites are disjoint. The situation is more intricate here because, while we still write a problem for each site, problems for different sites are not disjoint. Consequently, the construction of the coupling cannot be done directly by merging the results of  individual problems, as in the previously mentioned papers. Instead, additional problems for pairs of sites must be solved. The coupling is constructed  from the solutions to  flow problems for individual and paired sites.
 The reader interested in the mathematical theory and algorithms of network flow can refer to \cite{Ahu}.

The components of a network flow problem $P$, as posed here, are a set of nodes $\mathcal{N}$, a set of arcs  $\mathcal{A}\subseteq\mathcal{N}\times\mathcal{N}$ and two functions $l,u:\mathcal{A}\to\RR_+\cup\{+\infty \}$, representing the lower and upper bounds of capacities of arcs. Also, for convenience, we distinguish two nodes $O,Z\in\mathcal{N}$, called origin and destination, respectively. 

A solution $f$ to $P$  (also called flow on $ \mathcal{A}$) is a function  $f:\mathcal{A}\to\RR_+$ satisfying
\begin{enumerate}
	\item $l(v,w)\le f(v,w)\le u(v,w)$,  $\fora (v,w)\in \mathcal{A}$, 
	\item $\sum\limits_{w:(O,w)\in\mathcal{A}}\!\!f(O,w)=\sum\limits_{v:(v,Z)\in\mathcal{A}}\!\!f(v,Z)$ and 
	\item $\sum\limits_{v:(v,t)\in\mathcal{A}}\!\!f(v,t)=\sum\limits_{w:(t,w)\in\mathcal{A}}\!\!f(t,w)$, $\fora t\in\mathcal{N}\setminus\{O,Z\}$.
\end{enumerate}
 If  $P$ has a solution $f$ then $P$ is said to be feasible and $f(v,w)$ represents the flow on arc $(v,w)$. Further, for convenience, a flow $f$ is extended to a function on $\mathcal{N}\times\mathcal{N}$ by letting $f(v,w)=0$ for all $(v,w)\not\in\mathcal{A}$ and we assume that this is the case for all flow problems considered below. Finally, for $V,W\subseteq\mathcal{N}$ we define $f(V,W):=\sum_{v\in V,w\in W}f(v,w)$.

In all the flow problems in the following sections, the bound functions $l,u$ are shown in  a table, such as Table \ref{tab:P1xbounds}, and the set $\mathcal{A}$ of arcs is defined implicitly from such table. That is $(o,d)\in\mathcal{A}$ if and only if $(o,d)$ is listed in the table.

\subsection{The flow problem $P^x$}
Hereafter we fix $\eta,\xi\in \Omega$,  such that $\eta\le\xi$, and suppose that \eqref{eq:c1}, \eqref{eq:c2} hold  $\fora x\in \S$. Under such conditions, for each $x\in \S$, we define a flow problem  $P^x$ with nodes $\mathcal{N}^x$ and arcs $\mathcal{A}^x$. Let
\begin{equation*}
\begin{split}
S_1^{x}&=\{a\in C^{x}:\eta_a(x)<\eta(x)\},\quad T_1^{x}=S_1^{x}\cup R_1^{x},\\
S_2^{x}&=\{b\in C^{x}:\xi_b(x)>\xi(x)\},\quad T_2^{x}=S_2^{x}\cup R_2^{x}.
\end{split}
\end{equation*}

There is a node in $\mathcal{N}^x$ for each element of $T_1^x$; a node for each element  of $T_2^x$ and the two ``artificial'' nodes $O$ and $Z$. Since $T_1^x$ and $T_2^x$ need not be disjoint, we attach a label ``1'' to the elements of $T_1^x$ and a label ``2'' to the elements of $T_2^x$, in order to force a distinction. Therefore, an element of $T_1^x \cap T_2^x$ corresponds to two different nodes in $\mathcal{N}^x$: one with label ``1'' and another with label ``2''. More formally, the set of nodes may be defined as $\mathcal{N}^x=\{O,Z\}\cup(T_1^x\times\{1\})\cup (T_2^x\times\{2\})$.
Nonetheless, to avoid overburdening the notation, we refer to nodes $(a,1), (b,2)$ simply by $a,b$, without explicit mention to their labels, because they are clear from the context. For instance, when we allude to node $a\in T_1^x$ (resp. $b\in T_2^x$), we mean $(a,1)$ (resp. $(b,2)$). In this vein, we define the set of nodes of $P^x$ as 
\begin{equation}
\label{eq:nodes}
\mathcal{N}^{x}=\{O,Z\}\cup T_1^{x}\cup T_2^{x}.
\end{equation}
The above convention concerning nodes applies as well to the remaining flow problems $P^{xy}, P_1^{xy+}$, etc. in this section.

\label{pr:problem1}
\begin{defn}
	\label{def:P1x}
Let problem $P^x$ with nodes $\mathcal{N}^x$ defined in \eqref{eq:nodes} 
	and bounds (arcs $\mathcal{A}^x$) given in Table \ref{tab:P1xbounds}.
\end{defn}
\begin{table}[!h]
	\begin{center}
		{\caption{\label{tab:P1xbounds}Lower ($l$) and upper ($u$) bounds for arcs $(o,d)$ in problem $P^x$.}}
		{\scalebox{1}{
				\begin{tabular}{ccccc}
					\hline
					$o$& & $d$     & $l(o,d)$ & $u(o,d)$  \\
					\hline
					$O$& & $b\in S_2^{x}$     & 0  & $c_b^2(\xi)$   \\
					$O$& & $b\in R_2^{x}$     & $c_b^2(\xi)$  & $c_b^2(\xi)$   \\
					$b\in T_2^{x}$& &$a\in T_1^{x}$ and  $\eta_a\le\xi_b$&0&$\infty$\\
					$a \in R_1^{x}$& & $Z$     & $c_a^1(\eta)$  & $c_a^1(\eta)$   \\
					$a\in S_1^{x}$& &$ Z$     & 0  & $c_a^1(\eta)$   \\
					\hline
		\end{tabular}}}
	\end{center}
\end{table}

\begin{rem}
	\label{rem:condf}
 From Table \ref{tab:P1xbounds} we find  $\mathcal{A}^x\subseteq(\{O\}\times T_2^x)\cup (T_2^x\times T_1^x)\cup (T_1^x\times\{Z\})$. Note also that every solution $f^x$ satisfies  $f^x(O,b)=f^x(b, T_1^x), \fora b\in T_2^x$ and $f^x(a,Z)=f^x(T_2^x,a), \fora a\in T_1^x$.
	 Further, from Definition \ref{def:P1x} we have
\begin{equation}
\label{eq:problem1x}
\begin{split}
f^x(T_2^{x},a)=&c_a^1(\eta), \fora a\in R_1^{x},\quad f^x(T_2^{x},a)\le c_a^1(\eta),\fora a\in S_1^{x},\\
f^x(b,T_1^{x})=&c_b^2(\xi),\fora b\in R_2^{x},\quad f^x(b,T_1^{x})\le c_b^2(\xi),\fora b\in S_2^{x}.
\end{split}
\end{equation}
 Moreover, any function $\phi:T_2^x\times T_1^x\to\RR_+$ satisfying \eqref{eq:problem1x}, such that $\phi(b,a)=0$ if $\eta_a\not\le\xi_b$, can be  extended to a solution of $P^x$ by setting $\phi(O,b)=\phi(b,T_1^x)$,  $\fora b\in T_2^x$, and $\phi(a,Z)=\phi(T_2^x,a)$,  $\fora a\in T_1^x$.
\end{rem}
\begin{rem} For the construction of the coupling we need the values of $f^x(b,a)$, $\fora b\in T_2^x, a\in T_1^x$, with $\eta_a\le\xi_b$, such that \eqref{eq:problem1x} holds. We could obtain these quantities as a solution to a network flow problem, with supply nodes $T_2^x$ and demand nodes $T_1^x$, having lower and upper bounds on the nodes. The addition of the two artificial nodes $O, Z$ allows us to translate this problem, with bounds on the nodes, into a problem with bounds on the arcs, which is the standard formulation in the literature (see, e.g., \cite{FF}).
\end{rem}
\begin{prop}
	\label{propositionf1}
	$P^x$ is feasible.
\end{prop}

\begin{proof} Let $\{X,\overline X\}$ be a partition of $\mathcal{N}^{ x}$ such that either $O,Z\in X$ or $O,Z\in\overline X$. Let also 
$l(\overline X, X)$ be the sum of lower bounds of arcs $(\overline{x},x)$, with $\overline{x}\in \overline X$, $x\in X$, and let $u(X,\overline X)$ be the sum of upper bounds of arcs $(x,\overline{x})$, with $x\in X$, $\overline{x}\in\overline X$. Then, following the theorem in page 157 of \cite{Neu1984}, there is a flow from $O$ to $Z$ 	if and only if $l(\overline X,X)\le u(X,\overline X)$, for any partition as above.

	Take $\{X,\overline X\}$ such that $O,Z\in\overline X$. Then
	\begin{equation*}l(\overline X,X)=\sum_{b\in R_2^{x}\cap X}c_b^2(\xi).\end{equation*}
	For $u(X,\overline X)$ note that, if $b\in X$ and $a\in \overline X, a\ne Z$, such that $\eta_a\le \xi_b$, we have $u(X,\overline X)=+\infty$ and the condition $l(\overline X,X)\le u(X,\overline X)$ is trivially satisfied. We rule out this possibility and so, from here on, if $\eta_a\le \xi_b$, with $b\in X$, we suppose that $a\in X$. Thus,
	\begin{equation*}
	u(X,\overline X)=\sum_{a\in T_1^{ x}\cap X}c_a^1(\eta).
	\end{equation*}
	Finally observe that, letting $D_2=R_2^{x}\cap X$, we have 
	\begin{equation*}
	\sum_{b\in R_2^{x}\cap X}c_b^2(\xi)\le\sum_{a\in D_2^\downarrow}c_a^1(\eta)\le\sum_{a\in T_1^{x}\cap X}c_a^1(\eta),
	\end{equation*}
	where the first inequality follows from \eqref{eq:c2}  and the last  from the fact that $a\in D_2^\downarrow$ implies the existence of $b\in D_2\subseteq R_2^{x}$, such that $\eta_a\le\xi_b$. Therefore, since $\xi_b(x)<\eta(x)$, we have $\eta_a(x)<\eta(x)$ and so $a\in \S_1^{x}\subseteq T_1^{ x}$; also, $a\in X$ by the comments above.
	Finally, the case $O,Z\in X$ is proven analogously from \eqref{eq:c1}.
\end{proof}

We define sets of changes which are useful to describe properties of flows $f^x$. Comments in Remark \ref{rem:changes} are relevant here too.

\begin{defn}
	\label{def:partition}
	For $x, y\in \S$, such that $y\sim x$, let
\begin{equation*}
\begin{split}
&C_1^{+x}=\{a\in C^x:\eta_a(x)>\eta(x),\eta_a(z)=\eta(z),\,\fora z\ne x\},\\ 
&C_1^{-x}=\{a\in C^x:\eta_a(x)<\eta(x),\eta_a(z)=\eta(z),\,\fora z\ne x\},\\
&C_1^{+xy}=\{a\in C^x:\eta_a(x)>\eta(x), \eta_a(y)<\eta(y)\},\\
&C_1^{-xy}=\{a\in C^x:\eta_a(x)<\eta(x),\;\eta_a(y)>\eta(y)\},\\
&C_1^{+x\bullet}=\bigcup\limits_{z\sim x}C_1^{+xz}, \;C_1^{-x\bullet}=\bigcup\limits_{z\sim x}C_1^{-xz},\\
\end{split}
\end{equation*}
\begin{equation*}
\begin{split}
&C_2^{+x}=\{b\in C^x:\xi_b(x)>\xi(x),\xi_b(z)=\xi(z),\,\fora z\ne x\},\\
& C_2^{-x}=\{b\in C^x:\xi_b(x)<\xi(x),\xi_b(z)=\xi(z),\,\fora z\ne x\},\\
&C_2^{+xy}=\{b\in C^x:\xi_b(x)>\xi(x), \xi_b(y)<\xi(y)\},\\
&C_2^{-xy}=\{b\in C^x:\xi_b(x)<\xi(x),\;\xi_b(y)>\xi(y)\},\\
&C_2^{+x\bullet}=\bigcup\limits_{z\sim  x}C_2^{+xz}, \; C_2^{-x\bullet}=\bigcup\limits_{z\sim x}C_2^{-xz}.
\end{split}
\end{equation*}
\end{defn}
In the notation above, the exponent $e$ of $C_i^e$, with $e\in E:=\{\pm x, \pm xy, \pm x\bullet\}$,  informs about the changes in a set. For example, $C_1^{+x}$ has all changes that increase $\eta(x)$, leaving the remaining sites unchanged; $C_2^ {-xy}$ contains all changes that decrease $\xi(x)$ and increase $\xi(y)$ simultaneously;  $C_1^{-x\bullet}$ has all changes such that $\eta(x)$ decreases and $\eta(z)$ increases simultaneously, for some $ z\sim x$. 

The generator ${\cal L}_1$ of $(\eta_t)$ can now be written as
\begin{equation}
\label{eq:generatoreta2}
{\cal L}_1h(\eta)=\sum_{x\in \S}\left(\sum_{a\in C_1^{+x}}
c_a^1(\eta)(h(\eta_a)-h(\eta))+\!\sum_{a\in C_1^{-x}\cup C_1^{-x\bullet}}c_a^1(\eta)(h(\eta_a)-h(\eta))\right).
\end{equation}
Next we classify changes according to the possibility of altering the initial order $\eta\le\xi$.

\begin{defn}\label{def:goodbad}
	For any $e\in E$, let 
	\begin{equation*}
	\begin{split}
		G_1^e&=\{a\in C_1^e:\eta_a\le\xi\},\quad B_1^e=\{a\in C_1^e:\eta_a\not\le\xi\},\\
		G_2^e&=\{b\in C_2^e:\eta\le\xi_b\},\quad\, B_2^e=\{b\in C_2^e:\eta\not\le\xi_b\}.
	\end{split}
\end{equation*}
\end{defn}
	The sets $C_i^e$ can be written as union of  disjoint sets $G_i^e$ and $B_i^e$, for $i=1,2$. Note that
	\begin{equation}\label{goodbad}
R_1^{x}=B_1^{+x}\cup B_1^{+x\bullet},\, S_1^{x}=G_1^{-x}\cup G_1^{-x\bullet}\cup B_1^{-x\bullet},\quad R_2^{x}=B_2^{-x}\cup B_2^{-x\bullet},\, S_2^{x}=G_2^{+x}\cup G_2^{+x\bullet}\cup B_2^{+x\bullet}.
	\end{equation}
	We present some useful properties of  flows $f^x$.

\begin{lem}
	\label{lem:f1xtilde}  Let $f^x$ be a solution to $P^x$, then  
\begin{enumerate}
	\item[(a)] 	$f^x(R_2^{x}\cup B_2^{+x\bullet},B_1^{+x})=0$,
	\item[(b)] $f^x(B_2^{-x}, R_1^{x}\cup B_1^{-x\bullet})=0$,
	\item[(c)] $f^x(B_2^{-x\bullet},B_1^{+x\bullet})=0$,
	\item[(d)] 	\label{lem:f1x03} if $y,z\in \S$  \textup{s.t.}  $y\sim  x, z\sim  x$ and $z\ne y$, then $f^x(R_2^{x}\cup B_2^{+xz},B_1^{+xy})=0$,
	\item[(e)] \label{lem:f1x04} if $y,z\in \S$  \textup{s.t.}  $y\sim  x, z\sim  x$ and $z\ne y$, then 	$f^x(B_2^{-xy},R_1^{x}\cup B_1^{-xz})=0$.
\end{enumerate}
Moreover, there exists a solution $\tilde{f}^x$ such that
\begin{enumerate}
\item[(f)]
$\tilde{f}^x(S_2^{x}, S_1^{x})=0$.
\end{enumerate}
\end{lem}

\begin{proof}

(a)	Let $a\in B_1^{+x}$. If $b\in R_2^{x}$ then  $\eta_a(x)>\xi(x)\ge\eta(x)>\xi_b(x)$, so $\eta_a\not\le\xi_b$ and $f^x(b,a)=0$. If $b\in B_2^{+xy}\subseteq B_2^{+x\bullet}$, then $\xi_b(y)<\eta(y)=\eta_a(y)$. So, again, $\eta_a\not\le\xi_b$ and $f^x(b,a)=0$.

(b) 	Let $b\in B_2^{-x}$. If $a\in R_1^{x}$ then $\eta_a(x)>\xi(x)\ge\eta(x)>\xi_b(x)$, so $\eta_a\not\le\xi_b$ and $f^x(b,a)=0$. If $a\in B_1^{-xy}\subseteq B_1^{-x\bullet}$, then $\eta_a(y)>\xi(y)=\xi_b(y)$; so, $\eta_a\not\le\xi_b$ and $f^x(b,a)=0$.

(c) Let $a\in B_1^{+x\bullet}, b\in B_2^{-x\bullet}$, then $\eta_a(x)>\xi(x)$ and $\xi_b(x)<\eta(x)$. So $\eta_a\not\le\xi_b$ and $f^x(b,a)=0$.

 (d) 	Let $a\in B_1^{+xy}$. If $b\in R_2^{x}$  then $\eta_a(x)>\xi(x)$ and $\xi_b(x)<\xi(x)$, so $\eta_a\not\le\xi_b$ and $f^x(b,a)=0$. If $b\in B_2^{+xz}$ then $\xi_b(z)<\eta(z)=\eta_a(z)$, so $\eta_a\not\le\xi_b$ and $f^x(b,a)=0$.
 
 (e) 	Let $b\in B_2^{-xy}$. If $a\in R_1^{x}$  then $\eta_a(x)>\xi(x)$ and $\xi_b(x)<\xi(x)$, so $\eta_a\not\le\xi_b$ and $f^x(b,a)=0$. If $a\in B_1^{-xz}$, then $\eta_a(z)>\xi(z)=\xi_b(z)$, so $\eta_a\not\le\xi_b$ and $f^x(b,a)=0$.

 (f)  Let $f^x$ be a solution to $P^x$ and define $\tilde{f}^x:T_2^x\times T_1^x\to\RR^+$ as 
\begin{equation*}
\tilde{f}^x(b,a)=
\begin{cases} 
0\quad\qquad\text{if } a\in S_1^{x}, b\in S_2^{x},\\
f^x(b,a) \qquad\text{otherwise}.
\end{cases}
\end{equation*}
Let us see that  $\tilde{f}^x$ satisfies \eqref{eq:problem1x}. Indeed, by definition  $\tilde{f}^x$ satisfies the first and third conditions in \eqref{eq:problem1x} and also $\tilde{f}^x(T_2^{ x},a)=\tilde{f}^x(R_2^{ x},a)+\tilde{f}^x(S_2^{x},a)=f^x(R_2^{x},a)\le c_a^1(\eta)$, for all $a\in S_1^{x}$, which yields the second condition. The fourth is similarly obtained. By Remark \ref{rem:condf}, the conclusion follows.
\end{proof}

Hereafter we work with flows $f^x$ satisfying (f) in Lemma \ref{lem:f1xtilde}. We list some properties of such solutions in Table \ref{tab:tabla1x}, where entries show (necessary) conditions for arcs $(b,a)$ to have a positive flow. Note that the first two rows (columns) correspond to $R_2^{x}$ ($R_1^{x}$) while the remaining three correspond to $S_2^{x}$ ($S_1^{x}$). 
 \begin{table}[!h]
	\begin{center}
				{\caption{\label{tab:tabla1x} Necessary conditions for positive flow $f^x(b,a)$, with $b\in T_2^x, a\in T_1^x$. Symbol $-$ indicates 0 flow.}}
		{\scalebox{1}{
				\begin{tabular}{c|ccccc}
					\hline
					$T_2^x\backslash T_1^x$&$B_1^{+x}$&$B_1^{+x\bullet}$&$B_1^{-x\bullet}$&$G_1^{-x}$&$G_1^{-x\bullet}$\\ \hline
					$B_2^{-x}$&$-$&$-$&$-$&$\eta_a\le\xi_b$&$\eta_a\le\xi_b$\\ \hline
					$B_2^{-x\bullet}$&$-$&$-$&$\eta_a\le\xi_b$&$\eta_a\le\xi_b$&$\eta_a\le\xi_b;$\\ 
					&&&$\exists  y\sim x, a\in B_1^{-xy},b\in B_2^{-xy}$&&$ $\\ \hline
					$B_2^{+x\bullet}$&$-$&$\eta_a\le\xi_b$&$-$&$-$&$-$\\ 
					&&$\exists  y\sim x, a\in B_1^{+xy},b\in B_2^{+xy}$&&&$ $ \\ \hline
					$G_2^{+x}$&$\eta_a\le\xi_b$&$\eta_a\le\xi_b$&$-$&$-$&$-$\\ \hline
					$G_2^{+x\bullet}$&$\eta_a\le\xi_b$&$\eta_a\le\xi_b$&$-$&$-$&$-$\\
					\hline
		\end{tabular}}}
	\end{center}
\end{table}%

\subsection{The flow problem $P^{xy}$}\label{secprob2}
 We define flow problems for all pairs $(x,y)\in \mathcal{S}^2$, with $x\sim y$, such that $P^x, P^y$ have   respective solutions $f^x,f^y$,  assumed to satisfy  (f) of Lemma \ref{lem:f1xtilde}.
\begin{defn}
	\label{def:Txy} Let problem $P^{xy}$ have nodes $\mathcal{N}^{ xy}=\{O,Z\}\cup T_1^{xy}\cup T_2^{xy}$
and bounds shown in Table \ref{tab:P2xybounds},
 where \begin{equation*}
 	T_1^{xy} =G_1^{-x}\cup G_1^{-x\bullet}\cup G_1^{-y}\cup G_1^{-y\bullet}\cup B_1^{-xy}\cup B_1^{+xy}, \qquad T_2^{xy}= G_2^{+x}\cup G_2^{+x\bullet}\cup G_2^{+y}\cup G_2^{+y\bullet}\cup B_2^{-xy}\cup B_2^{+xy}.
 \end{equation*} 
 
\end{defn}
\begin{table}[!h]
	\begin{center}
			{\caption{\label{tab:P2xybounds}Lower ($l$) and upper ($u$) bounds for arcs $(o,d)$ in problem $P^{xy}$.}}
		{\scalebox{1}{
				\begin{tabular}{ccccc}
					\hline
					$o$& & $d$     & $l(o,d)$ & $u(o,d)$  \\
					\hline
					$O$& & $b\in G_2^{+x}\cup G_2^{+x\bullet}$     & 0  & $f^x(b,B_1^{+xy})$   \\
					$O$& & $b\in G_2^{+y}\cup G_2^{+y\bullet}$     & 0  & $f^y(b,B_1^{-xy})$   \\
					$O$& & $b\in B_2^{-xy}\cup B_2^{+xy}$     & $c_b^2(\xi)$  & $c_b^2(\xi)$   \\
					$b\in T_2^{xy}$& &$a\in T_1^{xy}$ and $\eta_a\le\xi_b$&0&$\infty$\\
					$a\in B_1^{+xy}\cup B_1^{-xy}$& & $Z$     & $c_a^1(\eta)$  & $c_a^1(\eta)$   \\
					$a\in G_1^{-x}\cup G_1^{-x\bullet}$& &$ Z$     & 0  & $f^x(B_2^{-xy},a)$   \\
					$a\in G_1^{-y}\cup G_1^{-y\bullet}$& &$ Z$     & 0  & $f^y(B_2^{+xy},a)$   \\
					\hline
		\end{tabular}}}
	\end{center}
\end{table}%
\begin{rem}
	Note that $P^{xy}=P^{yx}$. Observe also that there are changes belonging to several $T_1^{xy}$ (or $T_2^{xy}$). For instance, $a\in G_1^{-x}$ is in every $T_1^{xy}$ with $y\sim x$. Also, if $a\in G_1^{-xy}$ then, by definition,   $a\in T_1^{xz}$, $\fora z\sim  x$, but $a\not\in T_1^{yz}$,  for any $z\ne x$. Further,  if $a\in B_1^{+xy}$ then $a\in T_1^{xy}$ but $a\not\in T_1^{xz}$, for any $z\ne y$, and $a\not \in T_1^{yz}$, for any $z\ne x$. 
\end{rem}
The rest of this section is devoted to proving that $P^{xy}$ is feasible. To that end we consider auxiliary problems $P_1^{xy+}, P_1^{xy-}, P_2^{xy+}$ and $P_2^{xy-}$. 
\subsubsection{Problems $P_1^{xy+}$ and  $P_1^{xy-}$}\ 

\begin{defn}
\label{def:P1xy}

\noindent(a) Let  $P_1^{xy+}$ have nodes $\mathcal{N}_1^{ xy+}=\{O,Z\}\cup T_1^{xy+}\cup T_2^{xy+}$ and bounds  in the left panel of Table \ref{tab:P1xy+bounds}, where $T_1^{xy+}=G_1^{-x}\cup G_1^{-x\bullet}\cup B_1^{-xy}$ and $T_2^{xy+}=B_2^{-xy}$.\\

\noindent(b) Let  $P_1^{xy-}$ have nodes $\mathcal{N}_1^{ xy-}=\{O,Z\}\cup T_1^{xy-}\cup T_2^{xy-}$ and bounds in the right panel of Table \ref{tab:P1xy+bounds}, where $T_1^{xy-}=B_1^{+xy}$ and $T_2^{xy-}=G_2^{+x}\cup G_2^{+x\bullet}\cup B_2^{+xy}$.
\end{defn}

\begin{table}[!h]
	\begin{center}
				{\caption{\label{tab:P1xy+bounds}Lower ($l$) and upper ($u$) bounds for arcs $(o,d)$ in problems $P_1^{xy+}$ (left), $P_1^{xy-}$ (right).}} 
		{\scalebox{0.9}{
				\begin{tabular}{ccccc|}
					\hline
					$o$& & $d$     & $l(o,d)$ & $u(o,d)$  \\
					\hline
					$O$& & $b\in T_2^{xy+}$     & $c_b^2(\xi)$  & $c_b^2(\xi)$   \\
					$b\in T_2^{xy+}$& &
					$a\in T_1^{xy+}$ and $\eta_a\le\xi_b$ 
					&0&$\infty$\\
					$a\in G_1^{-x}\cup G_1^{-x\bullet}$& & $Z$     & 0  & $f^x(B_2^{-xy},a)$   \\
					$a\in B_1^{-xy}$& &$ Z$     & 0  & $c_a^1(\eta)$   \\
					\hline
		\end{tabular}
	\begin{tabular}{|ccccc}
		\hline
		$o$& & $d$     & $l(o,d)$ & $u(o,d)$  \\
		\hline
		$O$& & $b\in B_2^{+xy}$     & 0  & $c_b^2(\xi)$   \\
		$O$& & $b\in G_2^{+x}\cup G_2^{+x\bullet}$     & 0  & $f^x(b,B_1^{+xy})$   \\
		$b\in T_2^{xy-}$& &
		$a\in T_1^{xy-}$ and $\eta_a\le\xi_b$
		&0&$\infty$\\
		$a\in T_1^{xy-}$& & $Z$     & $c_a^1(\eta)$  & $c_a^1(\eta)$   \\
		\hline
\end{tabular}}}
	\end{center}
\end{table}%
\begin{prop}
	\label{prop:P1xy+feas}
	 $P_1^{xy+}$ and $P_1^{xy-}$ are feasible. 
\end{prop}
\begin{proof} Let $f_1^{xy+}(v,w)= f^x(v,w)$, for $v,w\in\mathcal{N}_1^{xy+}$. To prove that $f_1^{xy+}$ solves $P_1^{xy+}$ we check the following three conditions (see Table \ref{tab:P1xy+bounds}):
	\begin{enumerate}
		\item $f_1^{xy+}(b,T_1^{xy+})=c_b^2(\xi)$, $\fora b\in T_2^{xy+}=B_2^{-xy}$, 
		\item $f_1^{xy+}(T_2^{xy+},a)\le c_a^1(\eta)$, $\fora a\in B_1^{-xy}$, 
		\item $f_1^{xy+}(T_2^{xy+},a)\le f^x(B_2^{-xy},a)$,  $\fora a\in G_1^{-x}\cup G_1^{-x\bullet}$. 
	\end{enumerate}
	The first condition follows from $B_2^{-xy}\subseteq R_2^{x}$ together with the third line of  \eqref{eq:problem1x} and the second row of Table \ref{tab:tabla1x}. For the second, the  inclusion $B_2^{-xy}\subseteq T_2^x$ implies
	\begin{equation*}
	f_1^{xy+}(B_2^{-xy},a)= f^x(B_2^{-xy},a)
	\le  f^x(T_2^x,a),\fora a\in B_1^{-xy}.
	\end{equation*}
	Furthermore, by line 2 in \eqref{eq:problem1x} and the inclusion $B_1^{-xy}\subseteq S_1^{x}$, we have $ f^x(T_2^x,a)\le c_a^1(\eta)$. The third condition follows from the definitions, since $f_1^{xy+}(T_2^{xy+},a)= f^x(B_2^{-xy},a)$. Problem $P_1^{xy-}$ can be checked analogously.
\end{proof}

\subsubsection{Problems $P_2^{xy+}$ and $P_2^{xy-}$}\ 
\begin{defn}
\label{def:P2xy+-}	\noindent(a) Let $P_2^{xy+}$ have nodes $\mathcal{N}_2^{xy+}=\{O,Z\}\cup G_1^{-x}\cup G_1^{-x\bullet}\cup B_1^{-xy}\cup G_2^{+y}\cup G_2^{+y\bullet}\cup B_2^{-xy}$
and bounds in Table \ref{tab:P2xybounds}.\\
	
\noindent(b) Let $P_2^{xy-}$ have nodes $\mathcal{N}_2^{xy-}=\{O,Z\}\cup G_1^{-y}\cup  G_1^{-y\bullet}\cup  B_1^{+xy}\cup  G_2^{+x}\cup  G_2^{+x\bullet}\cup  B_2^{+xy}$
and bounds in Table \ref{tab:P2xybounds}.
\end{defn}
\begin{prop}
	\label{prop:P2xy+feas}
		$P_2^{xy+}$ and $P_2^{xy-}$ are feasible. 
\end{prop}
\begin{proof}
	We consider $P_2^{xy+}$ and 
	  apply  the theorem in \cite{Neu1984}. Let $\{X,\overline X\}$ be a partition of $\mathcal{N}_2^{xy+}$ such that $O,Z\in \overline X$. Then
	\begin{alignat*}{1}
	l(\overline X,X)&=\sum_{b\in B_2^{-xy}\cap X}c_b^2(\xi),\\
	u(X,\overline X)&=\sum_{a\in B_1^{-xy}\cap X}c_a^1(\eta)+\sum_{a\in (G_1^{-x}\cup G_1^{-x\bullet})\cap X}f^x(B_2^{-xy},a),
	\end{alignat*}
	assuming there is no arc $(b,a)$ with $a,b\notin \{O,Z\}$, such that $\eta_a\le\xi_b$ and $b\in X$, $a\in \overline X$, because otherwise $u(X,\overline X)=+\infty$ and condition $l(\overline X,X)\le u(X,\overline X)$ is trivial. 
	
	Let $\{X',\overline X'\}$ be the partition of $\mathcal{N}_1^{xy+}$ with $X'=X\cap \mathcal{N}_1^{xy+}$ and $\overline X'=\overline X\cap \mathcal{N}_1^{xy+}$. From the definition of $P_1^{xy+}$ we have
	
	\begin{alignat*}{1}
	l(\overline X',X')&=\sum_{b\in B_2^{-xy}\cap X'}c_b^2(\xi)=\sum_{b\in B_2^{-xy}\cap X}c_b^2(\xi),\\
	u(X',\overline X')&=\sum_{a\in B_1^{-xy}\cap X'}c_a^1(\eta)+\sum_{a\in (G_1^{-x}\cup G_1^{-x\bullet})\cap  X'}f^x(B_2^{-xy},a)\\
	&=\sum_{a\in B_1^{-xy}\cap X}c_a^1(\eta)+\sum_{a\in (G_1^{-x}\cup G_1^{-x\bullet})\cap  X}f^x(B_2^{-xy},a).
	\end{alignat*}
	So, since $P_1^{xy+}$ is feasible, we have $l(\overline X',X')\le u(X',\overline X')$  and so $l(\overline X,X)\le u(X,\overline X)$.
	
	Let now $\{X,\overline X\}$ be a partition of $\mathcal{N}_2^{xy+}$ such that $O,Z\in X$. Then 
	\begin{alignat*}{1}l(\overline X,X)&=\sum_{a\in B_1^{-xy}\cap \overline X}c_a^1(\eta),\\
	u(X,\overline X)&=\sum_{b\in B_2^{-xy}\cap\overline X}c_b^2(\xi)+\sum_{b\in (G_2^{+y}\cup G_2^{+y\bullet})\cap\overline X}f^y(b,B_1^{-xy}).
	\end{alignat*}
	As above, letting $\{X',\overline X'\}$ be the partition of $\mathcal{N}_1^{xy-}$, with  $X'=X\cap \mathcal{N}_1^{xy-}$ and $\overline X'=\overline X\cap \mathcal{N}_1^{xy-}$, it is easily checked that $l(\overline X',X')=l(\overline X,X)$ and $u(X',\overline X')=u(X,\overline X)$, hence  $P_2^{xy+}$ is feasible.	The feasibility of $P_2^{xy-}$ is  proved analogously. 
\end{proof}

\begin{prop} 
	\label{prop:P2xy+-} 
	$P^{xy}$ is feasible and has a solution $f^{xy}$ such that 
	\begin{equation}
	\label{eq:f2xycondition}
	f^{xy}(o,d)=0, \;\fora d\in G_1, o\in G_2,
	\end{equation} 
	with $G_1:= G_1^{-x}\cup G_1^{-x\bullet}\cup G_1^{-y}\cup G_1^{-y\bullet}$ and  $G_2:= G_2^{+x}\cup G_2^{+x\bullet}\cup G_2^{+y}\cup G_2^{+y\bullet}$.
\end{prop}

\begin{proof}
	Let $f_2^{xy+}$ and $f_2^{xy-}$ be solutions to $P_2^{xy+}$ and $P_2^{xy-}$ respectively. Note that, since $\mathcal{N}_2^{xy+}\cap\mathcal{N}_2^{xy-}=\{O,Z\}$, we can define
\begin{equation*}\label{deff2}
\tilde f^{xy}(o,d)= \begin{cases}f_2^{xy+}(o,d) & \mbox{if } o,d\in \mathcal{N}_2^{xy+},\\
f_2^{xy-}(o,d) & \mbox{if } o,d\in \mathcal{N}_2^{xy-},\\
0&\mbox{otherwise.}\end{cases}
\end{equation*}
It is clear that $\tilde f^{xy}$ is a flow from $O$ to $Z$ in the same network of $P^{xy}$ 
and, from the definitions of problems $P^{xy}$, $P_2^{xy+}$, $P_2^{xy-}$, it satisfies the bounds in Table \ref{tab:P2xybounds}.

Now, for $o, d\in\mathcal{N}^{xy}$ such that $o\ne O$, $d\ne Z$, let
\begin{equation*}
f^{xy}(o,d)= \begin{cases}0 & \mbox{if } d\in G_1, o\in G_2, \\
\tilde f^{xy}(o,d) &\mbox{otherwise.}\end{cases}
\end{equation*}
Let also $f^{xy}(O,b)=f^{xy}(b,T_1^{xy})$, for $b\in T_2^{xy}$, and $f^{xy}(a,Z)=f^{xy}(T_2^{xy},a)$,  for $a\in T_1^{xy}$. Then $f^{xy}$ is a solution to $P^{xy}$ satisfying \eqref{eq:f2xycondition}.
\end{proof}
In the construction of the OMC in Section \ref{sec:coc} we consider solutions $f^{xy}$ to $P^{xy}$, satisfying \eqref{eq:f2xycondition}. In Table \ref{tab:tabla2xy} we show conditions for a positive flow.

  \begin{table}[!h]
	\begin{center}
				{\caption{\label{tab:tabla2xy}Necessary conditions for positive flow $f^{xy}(b,a)$, with $b\in T_2^{xy}, a\in T_1^{xy}$. Symbol $-$ indicates 0 flow.}}
		{\scalebox{1}{
				\begin{tabular}{c|cccccc}
					\hline
					$b\in \backslash\; a\in $&$G_1^{-x}$&$G_1^{-y}$&$G_1^{-x\bullet}$&$G_1^{-y\bullet}$&$B_1^{+xy}$&$B_1^{-xy}$\\
					\hline
					$G_2^{+x}$&$-$&$ -$&$- $&$- $&$\eta_a\le\xi_b$&$- $\\ \hline
					$G_2^{+y}$&$- $&$-$&$- $&$- $&$- $&$\eta_a\le\xi_b $\\ \hline
					$G_2^{+x\bullet}$&$- $&$-$&$- $&$- $&$\eta_a\le\xi_b$&$- $\\ \hline
					$G_2^{+y\bullet} $&$- $&$-$&$- $&$- $&$- $&$\eta_a\le\xi_b $\\ \hline
					$B_2^{-xy}$&$\eta_a\le\xi_b $&$-$&$\eta_a\le\xi_b $&$- $&$- $&$\eta_a\le\xi_b $\\ \hline
					$B_2^{+xy}$&$-$&$\eta_a\le\xi_b $&$- $&$ \eta_a\le\xi_b$&$\eta_a\le\xi_b$&$- $\\
					\hline
		\end{tabular}}}
	\end{center}
\end{table}

\section{Construction of the coupling}\label{sec:coc}
\label{sec:firstcoupling}
We define a generator derived from the solutions to the network flow problems studied in Section \ref{tnfp}.  Proposition \ref{prop:gencoup} establishes that this generator is indeed the generator of an OMC of $(\eta_t)$ and $(\xi_t)$. This proves that conditions \eqref{eq:c1} and \eqref{eq:c2}  are sufficient for stochastic comparison. As a result, the proof of Theorem \ref{teorema} is concluded.
\begin{defn}
	\label{def:gencoup}
	Let the generator of  $(\eta_t',\xi_t')$,  acting on  $g:\{(\eta,\xi)\in\Omega\times \Omega:\ \eta\le\xi\}\to\mathbb{R}$,  be defined by
	\begin{equation}
	\label{eq:gen1}
	 \hspace{-200pt}{\cal L}_cg(\eta,\xi)=\sum_{x\in \S}\sum_{a\in G_1^{+x}}c_a^1(\eta)\left(g\left(\eta_a,\xi\right)-g(\eta,\xi)\right)
	\end{equation}
	\begin{equation}
	\label{eq:gen10}
	 \hspace{-160pt}+\sum_{x\in \S}\sum_{b\in G_2^{-x}}c_b^2(\xi)\left(g\left(\eta,\xi_b\right)-g(\eta,\xi)\right)
	\end{equation}
	\begin{equation}
	\label{eq:gen2}
	 \hspace{20pt}+\sum_{x\in \S}\sum_{a\in G_1^{-x}\cup G_1^{-x\bullet}}\Big(c_a^1(\eta)- f^x(B_2^{-x},a)-\sum_{
	y\sim x}f^{xy}(B_2^{-xy},a)\Big)
	\left(g\left(\eta_a,\xi\right)-g(\eta,\xi)\right)
	\end{equation}
	\begin{equation}
	\label{eq:gen11}
	 \hspace{15pt}+\sum_{x\in \S}\sum_{b\in G_2^{+x}\cup G_2^{+x\bullet}}\Big(c_b^2(\xi)- f^x(b,B_1^{+x})-\sum_{y\sim x}f^{xy}(b,B_1^{+xy})\Big)
	\left(g\left(\eta,\xi_b\right)-g(\eta,\xi)\right)
	\end{equation}
	\begin{equation}
	\label{eq:gen3}
	 \hspace{-100pt}+\sum_{x\in \S}\sum_{a\in G_1^{-x}\cup G_1^{-x\bullet}}\sum_{b\in B_2^{-x}} f^x(b,a)\left(g\left(\eta_a,\xi_b\right)-g(\eta,\xi)\right)
	\end{equation}
	\begin{equation}
	\label{eq:gen5}
	 \hspace{-105pt}+\sum_{x\in \S}\sum_{a\in B_1^{+x}}\sum_{b\in G_2^{+x}\cup G_2^{+x\bullet}} f^x(b,a)
	\left(g\left(\eta_a,\xi_b\right)-g(\eta,\xi)\right)
	\end{equation}
	\begin{equation}
	\label{eq:gen4}
	 \hspace{-80pt}+\sum_{x\in \S}\sum_{a\in G_1^{-x}\cup G_1^{-x\bullet}}\sum_{y\sim x}\sum_{b\in B_2^{-xy}}f^{xy}(b,a)
	\left(g\left(\eta_a,\xi_b\right)-g(\eta,\xi)\right)
	\end{equation}
	\begin{equation}
	\label{eq:gen6}
	\hspace{-80pt}+\sum_{x\in \S}\sum_{y\sim x}\sum_{a\in B_1^{-xy}}\sum_{b\in G_2^{+y}\cup G_2^{+y\bullet}}
	f^{xy}(b,a)\left(g\left(\eta_a,\xi_b\right)-g(\eta,\xi)\right)
	\end{equation}
	\begin{equation}
	\label{eq:genx}
	\hspace{-100pt}+\sum_{x\in \S}\sum_{y\sim x}\sum_{a\in B_1^{-xy}}\sum_{b\in B_2^{-xy}}
	f^{xy}(b,a)\left(g\left(\eta_a,\xi_b\right)-g(\eta,\xi)\right).
	\end{equation}
	\end{defn}

\begin{rem}
	Terms \eqref{eq:gen1} and \eqref{eq:gen10} in the definition of ${\cal L}_c$ above correspond to uncoupled arrivals (departures) in the first (second) component of individuals at site $x$, that do not break the order. Terms \eqref{eq:gen2} and \eqref{eq:gen11} are uncoupled changes of the first (second) component, corresponding to  departures or outbound migrations (arrivals or inbound migrations) of individuals at site $x$ in the first (second) component; the rates are the remainders of the original rates once these changes have been coupled with changes in the other component in \eqref{eq:gen3}-\eqref{eq:genx}. Terms \eqref{eq:gen3} and \eqref{eq:gen5} represent coupled changes of a departure (arrival) in the second (first) component with a departure or outbound migration (arrival or inbound migration) the the first (second) component. The $b\in B_2^{-x}$ ($a\in B_1^{+x}$) in the sum indicates that these departures (arrivals) would break the order if the component changed alone, so they need to be coupled with changes of the other component to maintain the order. Terms \eqref{eq:gen4} and \eqref{eq:gen6} represent coupled changes of an outbound migration from (inbound migration to) $x$ in the second (first) component, coupled with a departure or outbound migration (arrival or inbound migration) of the other component. Finally,  \eqref{eq:genx} represents coupled changes of migrations from $x$ to $y$, in both components. Note that $a\in B_1^{-xy}$ and $b\in B_2^{-xy}$, so both changes would break the order if the components were allowed to evolve independently.
\end{rem}

\begin{prop}
	\label{prop:gencoup}
	${\cal L}_c$ of Definition \ref{def:gencoup} is the generator of an OMC of  $(\eta_t)$ and $(\xi_t)$. 
\end{prop}
\begin{proof} 
	We first check that rates in \eqref{eq:gen2} and \eqref{eq:gen11} are non-negative.  For \eqref{eq:gen2} note, from row 6 of Table \ref{tab:P2xybounds}, that $f^{xy}(B_2^{-xy},a)\le f^{x}(B_2^{-xy},a)$,  $\fora x\in\mathcal{S},\ y\sim x,\ a\in G_1^{-x}\cup G_1^{-x\bullet}$. So,
	\begin{equation*}
\begin{split}
f^x(B_2^{-x},a)+\sum_{y\sim x}f^{xy}(B_2^{-xy},a)&\le f^x(B_2^{-x},a)+\sum_{y\sim x}f^{x}(B_2^{-xy},a)\\
&=f^x(B_2^{-x}\cup B_2^{-x\bullet},a)\\
&=f^x(R_2^{x},a)\\
&\le c_a^1(\eta),
\end{split}
	\end{equation*}
where the last inequality follows from the last row of Table \ref{tab:P1xbounds}. The argument is analogous for \eqref{eq:gen11}. 
Moreover, the order-preserving property follows from the definitions of the flows, because a positive flow on arc $(b,a)$  is possible only if $\eta_a\le\xi_b$.  

We now show that the marginals are distributed as the original processes and, to that end, we first let $g(\eta,\xi)=h(\eta)$ and check that ${\cal L}_cg(\eta,\xi)={\cal L}_1h(\eta)$. We proceed to examine all terms in Definition \ref{def:gencoup}. 
	
	Terms \eqref{eq:gen1} and \eqref{eq:gen5} yield, respectively,
	\begin{equation}
	\label{eq:gen1a}
	\sum_{x\in \S}\sum_{a\in G_1^{+x}}c_a^1(\eta)(h(\eta_a)-h(\eta))\quad\text{and}\quad \sum_{x\in \S}\sum_{a\in B_1^{+x}} f^x(G_2^{+x}\cup G_2^{+x\bullet},a)
	(h(\eta_a)-h(\eta)).
	\end{equation}
	Additionally, from the first line of display \eqref{eq:problem1x}, \eqref{goodbad} and the definition of $T_2^x$, we have
	\begin{equation*}
	 c_a^1(\eta)=f^x(T_2^x,a)=f^x(G_2^{+x}\cup G_2^{+x\bullet},a)+f^x(B_2^{-x}\cup B_2^{-x\bullet}\cup B_2^{+x\bullet},a),\fora x\in\mathcal{S}, a\in B_1^{+x}.
	\end{equation*}
	Also $f^x(B_2^{-x}\cup B_2^{-x\bullet}\cup B_2^{+x\bullet},a)=f^x(R_2^{x}\cup B_2^{+x\bullet},a)=0$, for $a\in B_1^{+x}$, from (a) in Lemma \ref{lem:f1xtilde}. Thus, we add both terms in \eqref{eq:gen1a} and obtain
	\begin{equation*}
	\sum_{x\in \S}\sum_{a\in G_1^{+x}\cup B_1^{+x}}c_a^1(\eta)(h(\eta_a)-h(\eta)),
	\end{equation*}
	which equals the first term of generator ${\cal L}_1$ in \eqref{eq:generatoreta2}, because $G_1^{+x}\cup B_1^{+x}=C_1^{+x}$ (see Definition \ref{def:partition}). Furthermore, adding \eqref{eq:gen2}, \eqref{eq:gen3} and \eqref{eq:gen4} results in
	\begin{equation*}
	\sum_{x\in \S}\sum_{a\in G_1^{-x}\cup G_1^{-x\bullet}}c_a^1(\eta)(h(\eta_a)-h(\eta)),
	\end{equation*}
	which is part of the second term in \eqref{eq:generatoreta2}, because  $C_1^{-x}\cup C_1^{-x\bullet}=G_1^{-x}\cup G_1^{-x\bullet}\cup B_1^{-x\bullet}, \fora x\in\S$. Finally, \eqref{eq:gen6} and \eqref{eq:genx} yield 
	\begin{equation}\label{equalityabove}
	\sum_{x\in \S}\sum_{y\sim x}\sum_{a\in B_1^{-xy}}
	f^{xy}(B_2^{-xy}\cup G_2^{+y}\cup G_2^{+y\bullet},a)(h(\eta_a)-h(\eta))=\sum_{x\in \S}\sum_{a\in B_1^{-x\bullet}}c_a^1(\eta)(h(\eta_a)-h(\eta)),
	\end{equation}
	which is the remaining part of the second term in \eqref{eq:generatoreta2}.
	To check \eqref{equalityabove} we have to show that $f^{xy}(B_2^{-xy}\cup G_2^{+y}\cup G_2^{+y\bullet},a)=c_a^1(\eta)$,   for all $x\in\mathcal{S}$, $y\sim x$ and $a\in B_1^{-xy}$. Note, from row 5 of Table \ref{tab:P2xybounds}, that $f^{xy}(T_2^{xy},a)=c_a^1(\eta)$,  $\fora a\in B_1^{-xy}$. Further, from Definition \ref{def:Txy} and the last column of Table \ref{tab:tabla2xy}, 
	\begin{equation*}
	f^{xy}(T_2^{xy},a)=f^{xy}(G_2^{+x}\cup G_2^{+x\bullet}\cup B_2^{+xy}\cup B_2^{-xy}\cup G_2^{+y}\cup G_2^{+y\bullet},a)
	=f^{xy}(B_2^{-xy}\cup G_2^{+y}\cup G_2^{+y\bullet},a).
	\end{equation*}
	The partial results above show that the first marginal process $(\eta_t')$ is distributed as $(\eta_t)$.
	
	For the second marginal $(\xi_t')$ we let $g(\eta,\xi)=h(\xi)$. The expression for the generator ${\cal L}_2$ of $(\xi_t)$ can be written in a form analogous to \eqref{eq:generatoreta2}, that is,
	\begin{equation*}
	\label{eq:generatorxi2}
	{\cal L}_2h(\xi)=\sum_{x\in \S}\left(\sum_{b\in C_2^{-x}}
	c_b^2(\xi)(h(\xi_b)-h(\xi))+\sum_{b\in C_2^{+x}\cup C_2^{+x\bullet}}c_b^2(\xi)(h(\xi_b)-h(\xi))\right).
	\end{equation*}
	Due to the symmetry of the generator of the coupling, the task of checking the second marginal is omitted.

The last step of the proof is to show that ${\cal L}_c$ is in fact the generator of an IPS on $(W\times W)^S$. To that end we show that conditions (2.2) and (2.4) in \cite{Pen} are fulfilled. Let $x\in S$ and note that in the above construction, a change at $x$ alone, in the first component, can be coupled with a change at $x$ and perhaps at another site $y\sim x$ in the second component. A change at $x$ and $y$ ($y\sim  x$), in the first component,  can be coupled either with a change at $x$ alone, at $y$ alone, at $x$  with $y$, or at $y$  with $z\sim  y$, in the second component. Therefore, the set of sites which change with $x$ or affect the change rate of $x$, is contained in the set $\{y\in\mathcal{S}:d_{\mathcal{S}}(x,y)\le2\delta\}$. Also, by the definition of coupling, the total rate involving site $x$ in ${\cal L}_c$ is bounded above by
$\sum_{a\in C^x}(c_a^1(\eta)+c_a^2(\xi))$,
which, by  \eqref{condexist2}, is uniformly bounded.	
\end{proof}

 \begin{rem}\label{noordenado} The generator in Definition \ref{def:gencoup} is defined only for $\eta\le\xi$, which
 suffices to establish stochastic ordering. An extension of the generator for $\eta\not\le\xi$ is straightforward. 
 \end{rem}

\begin{rem} Our construction of the coupling is not explicit, in the sense that it is written in terms of the values of the flows $f^x$, $f^{xy}$, which are not explicitly given. An explicit version of the coupling in a particular problem requires that the corresponding network flow problems be solved. This can be done either by  inspection, in simple problems, or by means of an algorithm, such as the max-flow min-cut algorithm for finite networks \cite{FF}, in more complex problems. Note that all the nodes in the network representing changes with rate zero can be deleted since there is no flow through them. This means that, in most situations, the network flow problems to be solved have a small number of nodes.  See Section \ref{excp} for three examples of construction of the coupling.
\end{rem}

\begin{rem}\label{objetivo}  The  flow problems in the construction of the coupling are posed as feasibility problems, meaning they have no specified objective function. It may be interesting to include such a function so that the solution minimizes a certain quantity. For instance, we might aim to keep the two components together as much as possible. This could be  achieved by incorporating an objective function that penalises departures from this situation. While we have not explored this possibility further, we believe it could prove useful when studying the limiting behaviour of the processes. For instance, in \cite{GS} the authors define an OMC such that discrepancies between the marginals are non-increasing. This property facilitates the finding of invariant measures in some examples.\end{rem}

\section{Examples}
 \label{secex}

In this section we illustrate the applicability of conditions \eqref{eq:c1} and \eqref{eq:c2}. 
We introduce and analyse a spatial population model allowing births, deaths and migrations of flocks between sites. This model extends the scope of those studied in \cite{Bor2012}.  Additionally, we explore conservative models, showing how our conditions reduce to those presented in \cite{GS} and \cite{GS23} for attractiveness. The section concludes with examples of  constructing the OMC. Throughout,  we adhere to the notation used in  \cite{Bor2012}, whenever feasible. 

Let $\S=\ZZ^d$ and $\Omega=W^\S$, with $W=\{0,1,\ldots,N\}$, for some $N\in\NN$. We take $\mathcal{V}(x)=\{y\in\mathcal{S}:\Vert x-y\Vert_1\le1\}$, that is, the site $x$ together with its $2d$ nearest neighbours. This means that migrations are possible only between neighbouring sites. We say that $(\zeta_t)$ follows model BDM (acronym for births, deaths, migration) if 
\begin{itemize}
	\item births $\zeta_t(x)\to\zeta_t(x)+1$ have rates $\zeta_t(x)\mathbf{1}_{\{\zeta_t(x)<N\}}$,
	\item deaths $\zeta_t(x)\to\zeta_t(x)-1$ have rates $$\zeta_t(x)\left(\boldsymbol{\phi}_A(\mathbf{r})\mathbf{1}_{\{\zeta_t(x)\le N_A\}}+\boldsymbol{\phi}(\mathbf{r})\mathbf{1}_{\{\zeta_t(x)>N_A\}}\right)+\boldsymbol{\mu}_1(\mathbf{r})\mathbf{1}_{\{\zeta_t(x)>N-M\}},$$
	\item catastrophic deaths $\zeta_t(x)\to\zeta_t(x)-k$ have rates $\boldsymbol{\mu}_k(\mathbf{r})\mathbf{1}_{\{\zeta_t(x)-k\ge N-M\}}$, $k\ge2$, and 
	\item migrations $(\zeta_t(x), \zeta_t(y))\to(\zeta_t(x)-k, \zeta_t(y)+l)$ have rates $$\boldsymbol{\boldsymbol{\lambda}}_{kl}(\mathbf{r})\mathbf{1}_{\{\zeta_t(x)-k\ge N-M,\zeta_t(y)+l\le N \}}, k, l\ge1, y\sim x,$$
\end{itemize} 
where $\mathbf{r}=\zeta_t^{\mathcal{V}(x)\setminus\{x\}}\in W^{2d}$ denotes the vector of values of $\zeta_t$ at the nearest neighbours of $x$. 
Parameters $N, N_A, M\in\NN$ are such that $N_A, M \le N$, while for each $\mathbf{r}\in\{0,1,\ldots,N\}^{2d}$, $\boldsymbol{\phi}(\mathbf{r}), \boldsymbol{\phi}_A(\mathbf{r})\in\RR_+$, $\boldsymbol{\mu}(\mathbf{r})\in\RR^M_+$  and $\boldsymbol{\Lambda}(\mathbf{r})$ is an $M\times N$ matrix with elements $\boldsymbol{\lambda}_{kl}(\mathbf{r})\in\RR_+$. That is, $\boldsymbol{\phi},\boldsymbol{\phi}_{A},\boldsymbol{\mu}$ and $\boldsymbol{\Lambda}$ are functions of $\mathbf{r}$. The BDM process $(\zeta_t)$ just described is said to have parameter vector  $(\boldsymbol{\phi},\boldsymbol{\phi}_{A},\boldsymbol{\mu},\boldsymbol{\Lambda},N,N_A,M)$.

Note that model BDM includes as particular cases models I, II and III, studied in \cite{Bor2012}. Indeed, model I, with births, deaths and migrations of single individuals, has parameters $\boldsymbol{\phi}=\boldsymbol{\phi}_A, \boldsymbol{\mu}=0$ and $\boldsymbol{\Lambda}$,  none depending on $\mathbf{r}$, with $\boldsymbol{\lambda}_{11}=\lambda/(2d)>0$ and $\boldsymbol{\lambda}_{kl}=0$,  $k>1$ or $l>1$. 
Model II, which incorporates the Allee effect, is the same as model I, with the exception that $\boldsymbol{\phi}\ne\boldsymbol{\phi}_{A}$. 
Finally, model III, allowing migrations of groups of individuals, is as model II, with the exception that 
$\boldsymbol{\Lambda}$ has elements $\boldsymbol{\lambda}_{kk}=\lambda/(2d)>0$, for $k\ge1$, and $\boldsymbol{\lambda}_{kl}=0$, for $k\ne l$.
Note that, while in our BDM model  the rates of change at site $x$ are affected by the values of the neighbouring sites, via the vector $\mathbf{r}$, the models in \cite{Bor2012} do not allow for this dependence.

The distinguishing features of model BDM include the incorporation of catastrophes, with rates given in vector $\boldsymbol{\mu}$, and non-conservative migrations, with rates contained in matrix $\boldsymbol{\Lambda}$. In population dynamics a catastrophe is a random event that results in the loss of a certain  number of individuals. Catastrophes and their consequences have been extensively analysed in various settings in mathematical biology; refer to \cite{junior2023extinction}, \cite{MT} and references therein. In a continuous-time Markovian setting, catastrophes and their impact in population evolution have been studied in several papers. For an extension of skip-free models  (birth and death models), see \cite{CP}. Non-conservative migrations, as referred to here, involve situations where the number of individuals departing from an origin site differs from the number reaching the destination site, given births and deaths may occur during migration. 
The concept of migration with mortality has received much attention in the literature of population dynamics, as can be seen, for example, in \cite{Arr03}, \cite{Gri01}, \cite{HAM}  and \cite{Zho04}.

We state conditions for stochastic ordering in the context of the BDM model. Note that such stochastic comparisons cannot be analysed using results in \cite{Bor}, even if the rates do not depend on $\mathbf{r}$, unless $\boldsymbol{\lambda}_{kl}=0$, for $k\ne l$,  because changes such as $\sigma_{xy}^{-k,l}$, with $k\ne l$, are not allowed in that paper. For $\mathbf{r}, \mathbf{s}\in\{0,\ldots,N\}^{2d}$, the expression $\mathbf{r}\le\mathbf{s}$ is understood componentwise.

\begin{prop}
	\label{prop:ejemplojavier} 
	Let $(\eta_t), (\xi_t)$ follow the  BDM model, with  parameter vectors  $(\boldsymbol{\phi}^i,\boldsymbol{\phi}^i_{A},\boldsymbol{\mu}^i,\boldsymbol{\Lambda}^i,N,N_A,M)$, $i=1,2$, respectively and let
	 $\boldsymbol{\lambda}_{i0}^1=\boldsymbol{\mu}_i^1/(2d)$, $\boldsymbol{\lambda}_{i0}^2=\boldsymbol{\mu}_i^2/(2d)$, $i=1,\ldots,M$. Then $(\eta_t)\le_{st}(\xi_t)$ if, for every $\mathbf{r}, \mathbf{s}$ such that $\mathbf{r}\le\mathbf{s}$, the following conditions hold
	
\begin{equation*}
(a)\quad\boldsymbol{\phi}^1(\mathbf{r})\ge\boldsymbol{\phi}^2(\mathbf{s}),\ \boldsymbol{\phi}^1_A(\mathbf{r})\ge\boldsymbol{\phi}^2_A(\mathbf{s}),
\end{equation*}
		\begin{equation}
	\label{eq:sumlambdale}
	(b)\quad\sum_{i=1}^m\boldsymbol{\lambda}_{ij}^1(\mathbf{r})\le\sum_{i=1}^m\boldsymbol{\lambda}^2_{ik}(\mathbf{s}),\, \fora m,k,j \textup{  s.t.  } 1\le m\le M, 1\le k\le j\le N,
	\end{equation}
	\begin{equation}
	\label{eq:sumlambdage}
	(c)\quad\sum_{j=0}^n\boldsymbol{\lambda}_{ij}^1(\mathbf{r})\ge\sum_{j=0}^n\boldsymbol{\lambda}_{lj}^2(\mathbf{s}),\,  \fora n,i,l \textup{  s.t.  } 0\le n\le N, 1\le i\le l\le M.
	\end{equation}

\end{prop}

\begin{proof}
Let $x\in\S$ and $\eta,\xi\in \Omega$, such that $\eta\le\xi$. Throughout the proof, all  rates $c_a^1(\eta)$ depend on $\mathbf{r}=\eta^{\mathcal{V}(x)\setminus\{x\}}$. However, as $\eta$ and $x$, and therefore $\mathbf{r}$, are fixed, we   omit $ \mathbf{r}$ for simplicity and write, for instance, $\boldsymbol{\lambda}_{ij}^1$ instead of $\boldsymbol{\lambda}_{ij}^1(\mathbf{r})$. The same applies for the second process.

  We first consider \eqref{eq:c1} and denote $a_1=\sigma_x^{+1}, a_{ij}=\sigma_{yx}^{-i,j}$, for $y\sim x$. Note that any $D_1\subseteq R_1^{x}$
 can be written either as 
  \begin{equation}
 \label{eq:D1}
D_1=\bigcup\limits_{y\sim x}\{a_{ij}:j\in J_y,i\in I_y(j)\}\quad\text{or}
\end{equation}
\begin{equation}
\label{eq:D1a1}
 D_1=\bigcup\limits_{y\sim x}\{a_{ij}:j\in J_y,i\in I_y(j)\}\cup\{a_1\},
\end{equation}
for some  $J_y\subseteq\{\xi(x)-\eta(x)+1,\ldots,N-\eta(x)\}$ and  $I_y(j)\subseteq\{1,\ldots,\eta(y)-N+M\}$.
Note also that, for  $j\in J_y$ and $i\in I_y(j)$, we have $c_{a_{ij}}^1(\eta)=\boldsymbol{\lambda}_{ij}^1$ and $ c_{a_{ik(j)}}^2(\xi)=\boldsymbol{\lambda}_{ik(j)}^2$, where $k(j)=j-\xi(x)+\eta(x)\le j$. 

Let $m_y(j)=\max I_y(j)$, for $j\in J_y$. Then, if $D_1$ is as in  \eqref{eq:D1}, we get
 \begin{equation}
 \label{eq:ineqD1}
 \begin{split}
 \sum_{a\in D_1}c_a^1(\eta)
 &=\sum_{y\sim x}\sum_{j\in J_y}\sum_{i\in I_y(j)}\boldsymbol{\lambda}^1_{ij}\le \sum_{y\sim x}\sum_{j\in J_y}\sum_{i=1}^{m_y(j)}
 \boldsymbol{\lambda}^1_{ij}\le \sum_{y\sim x}\sum_{j\in J_y}\sum_{i=1}^{m_y(j)}\boldsymbol{\lambda}_{ik(j)}^2\\
 &= \sum_{y\sim x}\sum_{j\in J_y}\sum_{i=1}^{m_y(j)}c_{a_{ik(j)}}^2(\xi)\le\sum_{b\in D_1^\uparrow}c_b^2(\xi).
 \end{split}
 \end{equation}
 The first inequality in \eqref{eq:ineqD1} follows from the inclusion $I_y(j)\subset\{1,\ldots,m_y(j)\}$ and the second one, from \eqref{eq:sumlambdale}. Also, the second equality  follows from $c_{a_{ik(j)}}^2(\xi)=\boldsymbol{\lambda}_{ik(j)}^2$, for $j\in J_y, 1\le i\le m_y(j)$. Further, for the last inequality note that 
\begin{equation*}
	\xi_{a_{ik(j)}}(x)=\xi(x)+j-\xi(x)+\eta(x)=\eta_{a_{m_y(j)j}}(x)\ {\rm and }\	\xi_{a_{ik(j)}}(y)=\xi(y)-i\ge \eta(y)-m_y(j)=\eta_{a_{m_y(j)j}}(y),
\end{equation*}
so $a_{ik(j)}\in D_1^\uparrow$, because $a_{m_y(j)j}\in D_1$, and condition \eqref{eq:c1} holds.

Finally, if $D_1$ is as in  \eqref{eq:D1a1}, the inequalities above remain valid but a term $c^1_{a_1}(\eta)=\eta(x)\mathbf{1}_{\{\eta(x)+1\le N \}}$ must be added in \eqref{eq:ineqD1}. Now, since  $\eta(x)=\xi(x)$ because $a_1\in D_1$, we have 
\begin{equation*}
	c^1_{a_1}(\eta)=\eta(x)\mathbf{1}_{\{\eta(x)+1\le N \}}=\xi(x)\mathbf{1}_{\{\xi(x)+1\le N \}}=c^2_{a_1}(\xi),
\end{equation*}
and so, $a_1\in D_1^\uparrow$. Hence, condition \eqref{eq:c1} holds in this case too.

 We proceed to check \eqref{eq:c2}. 
%
 Let  $b_l=\sigma_x^{-l}, b_{lj}=\sigma_{xy}^{-l,+j}$, for $x\in \mathcal{S}$ and $y\sim x$ and note that $D_2\subseteq R_2^{x}$ can be written as 
 \begin{equation}
 	\label{eq:D2}
 D_2=\{b_l:l\in L_2\}\cup \bigcup_{y\sim x}\{b_{lj},l\in L_y, j\in J_y(l)\},
 \end{equation}
 for some $L_2\subseteq\{\xi(x)-\eta(x)+1,\ldots,\xi(x)-N+M\}$, $L_y\subseteq \{\xi(x)-\eta(x)+1,\ldots,\xi(x)-N+M\}$ and  $J_y(l)\subseteq\{1,\ldots,N-\xi(y)\}$.  Also, for $l\in L_y\cup L_2$, let
\begin{equation*} 
J_y^0(l)= \begin{cases}J_y(l) &\mbox{if }l\in L_y\setminus L_2, \\
J_y(l)\cup\{0\} & \mbox{if }l\in L_2\cap L_y,\\
\{0\} & \mbox{if }l\in L_2\setminus L_y.
\end{cases}
\end{equation*}
Further, let $m_y(l)=\max J_y(l)$ and $m_y^0(l)=\max J_y^0(l)$.
 
 Suppose first that $b_1\not\in D_2$ (that is, $1\not\in L_2$) and observe that 
 \begin{equation}
 	\label{eq:cbl}
 	c_{b_l}^2(\xi)= \boldsymbol{\mu}_l^2=2d\boldsymbol{\lambda}_{l0}^2,\ \fora l\in L_2{\rm ;}\qquad  c_{b_{lj}}^2(\xi)=\boldsymbol{\lambda}_{lj}^2,\ \fora l\in L_y, j\in J_y(l).
 \end{equation} Then, from \eqref{eq:D2} and \eqref{eq:cbl},
\begin{equation}
\label{eq:sumbinD2}
\begin{split}
\sum_{b\in D_2}c_b^2(\xi)&=\sum_{l\in L_2}c_{b_l}^2(\xi)+\sum_{y\sim x}\sum_{l\in L_y}\sum_{j\in J_y(l)}c_{b_{lj}}^2(\xi)\\
&=\sum_{y\sim x}\Big(\sum_{l\in L_2}\boldsymbol{\lambda}_{l0}^2+\sum_{l\in L_y}\sum_{j\in J_y(l)}\boldsymbol{\lambda}_{lj}^2\Big).\\
\end{split}
\end{equation}
Moreover, 
\begin{equation}
\label{eq:sumlinL2}
\begin{split}
\sum_{l\in L_2}\boldsymbol{\lambda}_{l0}^2+\sum_{l\in L_y}\sum_{j\in J_y(l)}\boldsymbol{\lambda}_{lj}^2&=\sum_{l\in L_2\cap L_y}\boldsymbol{\lambda}_{l0}^2+\sum_{l\in L_2\setminus L_y}\boldsymbol{\lambda}_{l0}^2+\sum_{l\in L_y\cap L_2}\sum_{j\in J_y(l)}\boldsymbol{\lambda}_{lj}^2+\sum_{l\in L_y\setminus L_2}\sum_{j\in J_y(l)}\boldsymbol{\lambda}_{lj}^2\\
&=\sum_{l\in L_2\setminus L_y}\sum_{j\in J_y^0(l)}\boldsymbol{\lambda}_{lj}^2+\sum_{l\in L_y\cap L_2}\sum_{j\in J_y^0(l)}\boldsymbol{\lambda}_{lj}^2+\sum_{l\in L_y\setminus L_2}\sum_{j\in J_y^0(l)}\boldsymbol{\lambda}_{lj}^2\\
&\le\sum_{l\in L_2\cup L_y}\sum_{j\le m_y^0(l)}\boldsymbol{\lambda}_{lj}^2\\ &\le\sum_{l\in L_2\cup L_y}\sum_{j\le m_y^0(l)}\boldsymbol{\lambda}_{i(l)j}^1\\
&=\sum_{l\in L_2}\boldsymbol{\lambda}_{i(l)0}^1+\sum_{l\in L_y}\sum_{j=1}^{m_y(l)}\boldsymbol{\lambda}_{i(l)j}^1\\
&=\sum_{l\in L_2}\boldsymbol{\mu}_{i(l)}^1/(2d)+\sum_{l\in L_y}\sum_{j=1}^{m_y(l)}c_{b_{i(l)j}}^1(\eta),
\end{split}
\end{equation}
where the second inequality in display \eqref{eq:sumlinL2} follows from \eqref{eq:sumlambdage}, with $i(l)=l-\xi(x)+\eta(x)$, and the last equality follows from $c_{b_{i(l)j}}^1(\eta)=\boldsymbol{\lambda}_{i(l)j}^1$, for $l\in L_2\cup L_y, j\le m_y^0(l)$. Finally, from \eqref{eq:sumbinD2} and \eqref{eq:sumlinL2}, we obtain
\begin{equation}
\label{eq:sumbinD22}
\begin{split}
\sum_{b\in D_2}c_b^2(\xi)&\le\sum_{l\in L_2}\boldsymbol{\mu}_{i(l)}^1+\sum_{y\sim x}\sum_{l\in L_y}\sum_{j=1}^{m_y(l)}c_{b_{i(l)j}}^1(\eta)\le\sum_{a\in D_2^\downarrow}c_a^1(\eta).
\end{split}
\end{equation}
The last inequality in \eqref{eq:sumbinD22}  holds because $b_{i(l)}\in D_2^\downarrow$,  $\fora l\in L_2$ 
and $b_{i(l)j}\in D_2^\downarrow$,  $\fora l\in L_y, 1\le j\le m_y(l)$. 
 Indeed, $\fora l\in L_2$ we have 
 \begin{equation*}
 	\eta_{b_{i(l)}}(x)=\eta(x)-i(l)=\eta(x)-l+\xi(x)-\eta(x)=\xi(x)-l=\xi_{b_l}(x)
 \end{equation*} and so $b_{i(l)}\in D_2^\downarrow$. Next, $\fora l\in L_y, 1\le j\le m_y(l)$, we have
  \begin{equation*}
 \begin{split}
 	\eta_{b_{i(l)j}}(x)&=\eta(x)-l+\xi(x)-\eta(x)=\xi(x)-l=\xi_{b_{i(l)m_y(l)}}(x),\\
 	\eta_{b_{i(l)j}}(y)&=\eta(y)-j\le \xi(y)-m_y(l)=\xi_{b_{i(l)m_y(l)}}(y)
 \end{split}
 \end{equation*} and so, $b_{i(l)j}\in D_2^\downarrow$, since $b_{i(l)m_y(l)}\in D_2$. Therefore, from \eqref{eq:sumbinD22} we conclude that \eqref{eq:c2} holds.

Finally, if $b_1\in D_2$ (that is, $1\in L_2$), a term $\xi(x)(\boldsymbol{\phi}_A^2\mathbf{1}_{\{\xi(x)\le N_A\}}+\boldsymbol{\phi}^2\mathbf{1}_{\{\xi(x)>N_A}\})$ must be added to \eqref{eq:sumbinD2}. Since $b_1\in D_2$ implies $\eta(x)=\xi(x)$, we have
\begin{equation}
\label{eq:desphi}
\xi(x)\left(\boldsymbol{\phi}_A^2\mathbf{1}_{\{\xi(x)\le N_A\}}+\boldsymbol{\phi}^2\mathbf{1}_{\{\xi(x)>N_A}\}\right)\le
\eta(x)\left(\boldsymbol{\phi}^1_A\mathbf{1}_{\{\eta(x)\le N_A\}}+\boldsymbol{\phi}^1\mathbf{1}_{\{\eta(x)>N_A}\}\right).
\end{equation}
So, adding the RHS of \eqref{eq:desphi} to \eqref{eq:sumbinD22}, we see that \eqref{eq:c2} also holds in this case.
\end{proof}

 \begin{coro}
	\label{cor:attractive}
	Suppose $(\eta_t)$ follows the BDM model with  parameter  vector $(\boldsymbol{\phi},\boldsymbol{\phi}_{A},\boldsymbol{\mu},\boldsymbol{\Lambda},N,N_A,M)$ and let $\boldsymbol{\lambda}_{i0}=\boldsymbol{\mu}_i/(2d)$, for $i=1,\ldots,M$. Then $(\eta_t)$ is attractive if $\fora \mathbf{r}, \mathbf{s}$ such that $\mathbf{r}\le\mathbf{s}$,
		\begin{equation*}
			 \label{eq:comparaphi}
			 \hspace{-135pt}(a)\quad\boldsymbol{\phi}(\mathbf{r})\ge\boldsymbol{\phi}(\mathbf{s}),\quad\boldsymbol{\phi}_A(\mathbf{r})\ge\boldsymbol{\phi}_A(\mathbf{s}),
			 \end{equation*}
		\begin{equation}
		\label{eq:sumlambdale2}
		(b)\quad\sum_{i=1}^m\boldsymbol{\lambda}_{ij}(\mathbf{r})\le\sum_{i=1}^m\boldsymbol{\lambda}_{ik}(\mathbf{s}),\, \fora m,k,j \textup{  s.t.  } 1\le m\le M, 1\le k\le j\le N,
		\end{equation}
		\begin{equation}
		\label{eq:sumlambdage2}
		\hspace{-10pt}(c)\quad\sum_{j=0}^n\boldsymbol{\lambda}_{ij}(\mathbf{r})\ge\sum_{j=0}^n\boldsymbol{\lambda}_{lj}(\mathbf{s}),\,  \fora n,i,l \textup{  s.t.  } 0\le n\le N, 1\le i\le l\le M.
		\end{equation}
\end{coro}
In the following sections we analyse particular cases of the BDM model.

\subsection{Borrello's metapopulation models}
\label{modbor}
As commented above, models I, II and III, studied by Borrello in \cite{Bor2012}, are particular cases of our BDM, with $\boldsymbol{\lambda}_{kk}({\bf r})=\lambda/(2d)$, $\boldsymbol{\lambda}_{kl}({\bf r})=0$ for $k\ne l$ and $\boldsymbol{\mu}_k({\bf r})=0$, for $1\le k\le M$. 
	Propositions 3.1, 4.1 and 5.1 in \cite{Bor2012}, relative to stochastic comparisons and attractiveness of such models, follow from our Proposition \ref{prop:ejemplojavier} and Corollary \ref{cor:attractive}.
	
\subsection{Migration rate depending on the flock size and catastrophes}\label{modmsd}
In \cite{Bor2012}, the transition rates of migrations  are assumed to be independent of the flock size and of the neighbouring sites. Here, we relax this condition by letting $\boldsymbol{\lambda}_{kk}({\bf r})=\boldsymbol{\lambda}_k({\bf r})$. Migration processes where rates depend on the flock size (without allowing births or deaths) have been studied in \cite{FAJ2016}. 
Catastrophes are included in the BDM model by letting $\boldsymbol{\mu}_k$ be positive. The remaining parameters are as in Section \ref{modbor}.

\begin{prop}
	\label{MSDC}  
	Let $(\eta_t), (\xi_t)$ follow the  BDM model, with  parameter vectors  $(\boldsymbol{\phi}^i,\boldsymbol{\phi}^i_{A},\boldsymbol{\mu}^i,\boldsymbol{\Lambda}^i,N,N_A,M)$, $i=1,2$, respectively, where $\boldsymbol{\lambda}^i_{kk}=\boldsymbol{\lambda}^i_k$ and $\boldsymbol{\lambda}^i_{kl}=0$, for $k\ne l$. Then $(\eta_t)\le_{st}(\xi_t)$ if $\fora \mathbf{r}, \mathbf{s}$ such that $\mathbf{r}\le\mathbf{s}$,
	\begin{enumerate}
	\item[(a)] 	$\boldsymbol{\phi}^1({\bf r})\ge\boldsymbol{\phi}^2({\bf s})$, $\boldsymbol{\phi}^1_A({\bf r})\ge\boldsymbol{\phi}^2_A({\bf s})$,
		\item[(b)] $\boldsymbol{\lambda}_j^1({\bf r})\le\boldsymbol{\lambda}_k^2({\bf s})$, $\fora k,j$ \textup{s.t.}  $1\le k\le j\le M$,
	\item[(c)] $\boldsymbol{\mu}_i^1({\bf r})\ge\boldsymbol{\mu}_l^2({\bf s})$; $\boldsymbol{\mu}_i^1({\bf r})+2d\boldsymbol{\lambda}_i^1({\bf r})\ge\boldsymbol{\mu}_l^2({\bf s})+2d\boldsymbol{\lambda}_l^2({\bf s})$, $\fora i,l$ \textup{ s.t.}  $1\le i\le l\le M$.
\end{enumerate}
\end{prop}

\begin{proof} Note that $\boldsymbol{\lambda}_{kl}^i({\bf r})=\boldsymbol{\lambda}_k^i({\bf r})\mathbf{1}_{\{k=l\}}$, for $k=1,\ldots,M$, $l=1,\ldots,N$ and $\boldsymbol{\lambda}_{k0}^i({\bf r})=\boldsymbol{\mu}^i_k({\bf r})/(2d)$, for $k=1,\ldots,M$, $i=1,2$. Then, condition \eqref{eq:sumlambdale} can be written as
\begin{equation*} 
\boldsymbol{\lambda}_j^1({\bf r})\mathbf{1}_{\{j\le m\}}\le\boldsymbol{\lambda}_k^2({\bf s})\mathbf{1}_{\{k\le m\}},\quad \fora  {\mathbf r}, {\mathbf s}, k,j,m \textup{ s.t. } {\mathbf r}\le{\mathbf s},\ 1\le k\le j\le M, 1\le m\le M,
\end{equation*}
which is clearly implied by $(b)$. Analogously, \eqref{eq:sumlambdage} can be written as $\boldsymbol{\mu}_i^1({\bf r})/(2d)\ge\boldsymbol{\mu}_l^2({\bf s})/(2d)$, for $n=0$, and
\begin{equation*}\frac{\boldsymbol{\mu}_i^1({\bf r})}{2d}+\boldsymbol{\lambda}_i^1({\bf r})\mathbf{1}_{\{i\le n\}}\ge\frac{\boldsymbol{\mu}_l^2({\bf s})}{2d}+\boldsymbol{\lambda}_l^2({\bf s})\mathbf{1}_{\{l\le n\}}, \quad \fora  {\mathbf r}, {\mathbf s}, i,l,m \textup{ s.t. } {\mathbf r}\le{\mathbf s},\ 1\le i\le l\le M, 1\le n\le M,
\end{equation*}
which holds by condition $(c)$. The result follows from Proposition \ref{prop:ejemplojavier}.
\end{proof}

\begin{coro}\label{cor:atra}
Let $(\eta_t)$ follow the  BDM model, with  parameter vector $(\boldsymbol{\phi},\boldsymbol{\phi}_{A},\boldsymbol{\mu},\boldsymbol{\Lambda},N,N_A,M)$, where $\boldsymbol{\lambda}_{kk}=\boldsymbol{\lambda}_k$ and $\boldsymbol{\lambda}_{kl}=0$, for $k\ne l$. Then $(\eta_t)$ is attractive if, $\fora \mathbf{r}, \mathbf{s}$ such that $\mathbf{r}\le\mathbf{s}$,
 \begin{enumerate}
	\item[(a)] $\boldsymbol{\phi}(\mathbf{r})\ge \boldsymbol{\phi}(\mathbf{s})$, 
	 $\boldsymbol{\phi}_A(\mathbf{r})\ge \boldsymbol{\phi}_A(\mathbf{s})$,
	 \item[(b)]  $\boldsymbol{\lambda}_i(\mathbf{r})\ge\boldsymbol{\lambda}_j(\mathbf{s}), \fora i,j$ \textup{ s.t.}  $1\le i\le j\le M$,
	 \item[(c)] $\boldsymbol{\mu}_i(\mathbf{r})\ge\boldsymbol{\mu}_j(\mathbf{s}), \,\fora i,j$ \textup{ s.t.}  $1\le i\le j\le M$.
\end{enumerate}
\end{coro}
\begin{proof} The conclusion is obtained directly from Proposition \ref{MSDC}.
\end{proof}

\subsection{Allee effect in migrations and deaths}
\label{sec:allee} 
 As a particular instance of the processes in Section \ref{modmsd}, we consider the BDM with $\boldsymbol{\lambda}_j(\mathbf{r})=\lambda_j+\lambda_A\mathbf{1}_{\{S(\mathbf{r})\ge A\}}$ and 
$\boldsymbol{\mu}_j(\mathbf{r})=\mu_j+\mu_A\mathbf{1}_{\{S(\mathbf{r})< A\}}$, $j=1,\ldots,M$, where $S(\mathbf{r})$ represents  the sum of components of $\mathbf{r}$ and $\lambda_j, \lambda_A, \mu_j, \mu_A\in\RR_+, j=1,\ldots,M$. That is, the migration and catastrophe rates depend on the whole neighbourhood and not only on the modified sites. We assume that parameters $\boldsymbol{\phi}$ and $\boldsymbol{\phi}_A$ do not depend on $\mathbf{r}$.

Similar to the classical Allee effect, in which individuals are more susceptible to mortality in low population conditions, deaths in this model  are more likely to occur when the surrounding sites are less crowded. Additionally, densely populated neighbourhoods are more attractive, leading to a higher migration rate when there are more individuals in neighbouring sites.

Importantly, in this model the death rates depend non-linearly on the values of the neighbours. Furthermore, migration rates are influenced not only by the departure and arrival sites but also by the neighbourhood of the departure site, once again in a non-linear manner. These situations cannot be handled using the results in \cite{Bor} or \cite{Bor2012}.

\begin{prop}\label{propejnuevo}
The process with Allee effect in migrations and deaths is attractive if

\begin{enumerate}
	\item[(a)] 	$2d\lambda_A\le\mu_A$,
	\item[(b)] 	$\lambda_j\le\lambda_k$, $\fora k,j$\textup{  s.t.}  $1\le k\le j\le M$,
		\item[(c)] $\mu_i\ge\mu_l$, $\mu_i+2d\lambda_i\ge\mu_l+2d\lambda_l$, $\fora i,l$ \textup{ s.t. }   $1\le i\le l\le M$.
\end{enumerate}
\end{prop}

\begin{proof}
We apply Corollary \ref{cor:atra}. Condition $(a)$ in that corollary is trivial since $\boldsymbol{\phi}$ and $\boldsymbol{\phi}_A$ do not depend on $\mathbf{r}$. Condition $(b)$ in the corollary follows from $(b)$ of this proposition, by noting that $\mathbf{1}_{\{S(\mathbf{r})\ge A\}}\le\mathbf{1}_{\{S(\mathbf{s})\ge A\}}$, if $\mathbf{r}\le\mathbf{s}$.

For the first part of condition $(c)$ in the corollary, we use $\mu_i\ge\mu_l$, for $i\le l$, and 
$\mathbf{1}_{\{S(\mathbf{r})< A\}}\ge\mathbf{1}_{\{S(\mathbf{s})< A\}}$. For the second part, note that
\begin{equation}
	\label{comparaej}
	\boldsymbol{\mu}_i(\mathbf{r})+2d\boldsymbol{\lambda}_i(\mathbf{r})=\mu_i+\mu_A\mathbf{1}_{\{S(\mathbf{r})< A\}}+2d\lambda_i+2d\lambda_A\mathbf{1}_{\{S(\mathbf{r})\ge A\}}
\ge\mu_l+2d\lambda_l+\mu_A\mathbf{1}_{\{S(\mathbf{r})<A\}}+2d\lambda_A\mathbf{1}_{\{S(\mathbf{r})\ge A\}}.\end{equation}
Since $\mathbf{r}\le\mathbf{s}$, we have $S(\mathbf{r})\le S(\mathbf{s})$. Now, if either $S(\mathbf{r})\ge A$ or $S(\mathbf{s})<A$, the RHS of \eqref{comparaej} if equal to $\boldsymbol{\mu}_l(\mathbf{s})+2d\lambda_l(\mathbf{s})$. If, on the contrary, $S(\mathbf{r})<A$ and $S(\mathbf{s})\ge A$, then the RHS of \eqref{comparaej} is $\mu_l+2d\lambda_l+\mu_A\ge\mu_l+2d\lambda_l+2d\lambda_A=\boldsymbol{\mu}_l(\mathbf{s})+2d\boldsymbol{\lambda}_l(\mathbf{s})$ and the result is proved.
\end{proof}
\begin{rem}
	Note that in the particular case of $\lambda_j=\lambda$, $\mu_j=\mu$,  $\fora j\in\{1,\ldots,M\}$,  the process is attractive if $2d\lambda_A\le\mu_A$.
\end{rem}

\subsection{Nonconservative migrations}
Consider a situation in which each individual dies with probability $1-p$ during a migration, independently of other individuals. Consequently, if $k$ individuals depart from site $x$ to $y$,  a binomially distributed number of them reaches $y$. For simplicity we assume that the migration rate does not depend on the flock size or the other neighbours. This yields

\begin{equation}
\label{eq:lambdabinom}
\boldsymbol{\lambda}_{kl}=\lambda\tbinom{k}{l}p^l(1-p)^{k-l}, 
\end{equation}
for $1\le k\le M, 1\le l\le k$, $\lambda>0, p\in[0,1]$, $\boldsymbol{\lambda}_{kl}=0$, for $k<l$, and $\boldsymbol{\mu}_k=2d\lambda(1-p)^k$, for $k\ge1$. 
This model is an extension of model III in \cite{Bor2012}, where no deaths in migrations are considered; in other words, model III is equivalent to the current model when $p=1$. In \cite{Bor2012} the author initially proves that the model is attractive (Proposition 5.1) and then gives conditions for extinction and survival (Theorem 5.1). Note that the attractiveness of our model cannot analysed using the techniques presented in \cite{Bor} or in \cite{Bor2012}, as these works do not account for non-conservative migrations. In what follows we use the results in this paper to prove that the model is attractive. Additionally, we give conditions for extinction and survival in the line of Theorem 5.1 of \cite{Bor2012}. 

\begin{prop}
	\label{prop:binomattract}
 The BDM model with parameter described in \eqref{eq:lambdabinom} is attractive.	
\end{prop}
\begin{proof}
	We show that the parameters above satisfy conditions \eqref{eq:sumlambdale2} and \eqref{eq:sumlambdage2} of Corollary \ref{cor:attractive}.
	
	For \eqref{eq:sumlambdale2} it suffices to take $1\le k<m$ and $j=k+1$. Then we have to prove  
	\begin{equation*}
\sum_{i=k+1}^{m}\tbinom{i}{k+1}p^{k+1}q^{i-k-1}\le \sum_{i=k}^{m}\tbinom{i}{k}p^{k}q^{i-k},
	\end{equation*}
	with $q=1-p$. To that end notice that the difference between the RHS and the LHS above (denoted $D$) satisfies
	\begin{equation*}
	\begin{split}
	\Big(\tfrac{q}{p}\Big)^kD&=\sum_{i=k}^{m}\tbinom{i}{k}q^{i}-p\sum_{i=k+1}^{m}\tbinom{i}{k+1}q^{i-1}\\
	&=\sum_{i=k+1}^{m}\tbinom{i}{k+1}q^i+\sum_{i=k}^{m}\tbinom{i}{k}q^i-\sum_{i=k+1}^{m}\tbinom{i}{k+1}q^{i-1}\\
	&=\sum_{i=k+1}^{m}\tbinom{i}{k}q^i+\sum_{i=k+1}^{m}\tbinom{i}{k+1}q^{i}-\sum_{i=k+2}^{m}\tbinom{i}{k+1}q^{i-1}\\
	&=\sum_{i=k+1}^{m}\left(\tbinom{i}{k}+\tbinom{i}{k+1}\right)q^i-\sum_{i=k+1}^{m-1}\tbinom{i+1}{k+1}q^{i}\\
	&=\tbinom{m+1}{k+1}q^m\ge0.
	\end{split}
	\end{equation*}
	Analogously for \eqref{eq:sumlambdage2} and noting that $\boldsymbol{\lambda}_{k0}=\lambda(1-p)^k$, we have to show that
	\begin{equation*}
	\sum\limits_{j=0}^{n}\tbinom{i}{j}p^jq^{i-j}\ge \sum\limits_{j=0}^{n}\tbinom{i+1}{j}p^jq^{i+1-j},\fora n,i \textup{ s.t. }0\le n\le M, 1\le i<M,
	\end{equation*}
	which is equivalent to $P(B_i\le n)\ge P(B_{i+1}\le n)$, where $B_i$ is a binomial random variable with parameters $i, p$. Such inequality follows from the obvious stochastic domination $B_i\le_{st}B_{i+1}$.
\end{proof}

For the next result we borrow some notation from \cite{Bor2012}, both in the statement and in the proof. 
\begin{theorem}
	\label{teoremaborrello} 
	Let $d\ge2$. For all $\lambda>0$, $N_A\ge0$, $p\in(0,1)$:
\begin{itemize}
	\item[(a)] if $\phi<1$, there exists $N_c(\phi,\lambda,N_A)$ such that for each $N>N_c(\phi,\lambda,N_A)$, there exists $M(N_A)$ so that the process starting from $\eta_0$ with $\vert\eta_0\vert\ge1$ has a positive probability of survival for each $\phi_A<\infty$. Moreover if $\eta_0\in\overline\Omega_N$ the process converges to a nontrivial invariant measure for each $\phi_A<\infty$;
\item[(b)] if $\phi\ge1$, the process becomes extinct $\fora N, \lambda, \phi_A>1$, $M$ and for any finite initial configuration. If $\eta_0\in\Omega_N$ is not finite the process becomes extinct if $\phi>1$.
\end{itemize}
\end{theorem}

\begin{proof} The proof follows the lines of the proof of Theorem 5.1 in \cite{Bor2012}. For $(a)$, the construction in that proof for the process with no deaths in migrations, can be also made for our process. Lemmas 7.5 and 7.6 in that paper hold without any changes in our case. Lemma 7.7 also holds, but the proof must be modified slightly. 
	
	With the notation there, when $\xi_t$ visits $N$, there is a migration of $M$ individuals  from 0  to $y$, all of them reaching $y$, with rate $\lambda p^N/(2d)$, while other possibilities (death of an individual at 0, migration of a smaller number of individuals to $y$, migration of $M$ individuals to $y$ with some of them dying during the migration or migration to a different site $z$) occur at a rate smaller than $N\phi+\lambda M\frac{2d-1}{2d}+\frac{\lambda}{2d}$. Thus, every time $\xi_t$ visits $N$, with a probability larger than $\frac{\lambda p^N}{2d(\lambda M+N\phi)}$, there are $M$ individuals actually reaching $y$ from 0. Following the reasoning in Lemma 7.7 of \cite{Bor2012}, this bound is enough to conclude that the Lemma also holds in our case. The rest of the proof of Theorem 5.5 in \cite{Bor2012}, including the convergence to the invariant measure from $\eta_0\in\overline \Omega_N$, which exists by attractiveness, carries over to our setting.

For the converse, given that  our process is attractive, the same proof of Theorem 5.1 in \cite{Bor2012} (which refers to the proof of Theorem 3.2 [step(iii)] in the same paper) holds here with minimal changes.
\end{proof}

\subsection{General exclusion processes}\label{exGep}

In exclusion processes, each site $x\in\S$ can be empty or occupied by a particle, thus $W=\{0,1\}$. Here $\eta(x)=0$ indicates that  site $x$ is empty, while $\eta(x)=1$ means that it is occupied. Following the notation in \cite{GS23}, particles move from a site $x$ to another (empty) site $y$ at rate $\Gamma_\eta(x,y)$.  As there are no arrivals of individuals, these processes are not particular cases of BDM.  Exclusion processes have been extensively studied in the literature of IPS; for definitions, refer to  \cite{Lig},  and for recent results, consult \cite{Ayyer2023}. 

The most studied version of this model is the simple exclusion model, with rates  $\Gamma_\eta(x,y)$ given by  $\eta(x)(1-\eta(y))p(x,y)$, where $p(x,y)$ is a transition matrix. 
Conditions for the attractiveness of simple exclusion processes and existence of an OMC are well known. 
These properties, in turn, are used, for instance, to find the invariant measures of the processes.

For general exclusion processes, where $\Gamma_\eta(x,y)$ may depend on $\eta$ in an arbitrary manner, the question of attractiveness and OMC have been investigated in \cite{FF97} and, more recently, in \cite{GS23}. In Theorem 2.4 of \cite{GS23}, necessary and sufficient conditions for attractiveness are  derived through the construction of an OMC (see Propositions 3.2 and 3.3). This coupling has the interesting property of  minimizing discrepancies between its two components and is employed to analyse the invariant measures of several exclusion processes. 

It is noteworthy that, since $W$ is finite, we do not  rely on \cite{Pen} for the existence of the processes, and as a result, the  conditions in Section \ref{sec:existence} are not needed. Instead, we can assume condition (2.4) in \cite{GS23}, which, in particular, permits particle jumps to arbitrarily distant sites. Consequently, our Theorem \ref{teorema} can be applied in the context of \cite{GS23}.  

Following Theorem 2.4 in \cite{GS23}, the necessary and sufficient conditions on  $\Gamma_\eta(x,y)$ for a process to be attractive are: $\fora\ \eta\le\xi,\ x\in\S$,

\begin{alignat}{1}
	\label{eq:gs1}
	\sum_{y\in\S}\eta(y)\left(\Gamma_\eta(y,x)-\Gamma_\xi(y,x)\right)^+&\le\sum_{y\in\S}\xi(y)(1-\eta(y))\Gamma_\xi(y,x), \,\text{ if }\xi(x)=0,\\
	\label{eq:gs2}
	\sum_{y\in\S}(1-\xi(y))\left(\Gamma_\xi(x,y)-\Gamma_\eta(x,y)\right)^+&\le\sum_{y\in\S}\xi(y)(1-\eta(y))\Gamma_\eta(x,y),  \,\text{ if }\eta(x)=1.
\end{alignat}
Although Theorem 2.4 in \cite{GS23} and our Theorem  \ref{teorema} have established the equivalence between  the two conditions above and our conditions \eqref{eq:c1} and \eqref{eq:c2}, it is instructive to independently confirm this equivalence.

Let us verify that, indeed, our condition  \eqref{eq:c1} reduces to \eqref{eq:gs1} for this class of processes. Given $\eta\le\xi$, we define the sets $N_0^0=\{x\in\S:\eta(x)=\xi(x)=0\}$, $N_0^1=\{x\in\S:\eta(x)=0,\xi(x)=1\}$ and $N_1^1=\{x\in\S:\eta(x)=\xi(x)=1\}$. As $\S=N_0^0\cup N_0^1\cup N_1^1$, condition \eqref{eq:gs1} can  be expressed as
\begin{equation}\label{210gs} \sum_{y\in N_1^1}\left(\Gamma_\eta(y,x)-\Gamma_\xi(y,x)\right)^+\le\sum_{y\in N_0^1}\Gamma_\xi(y,x), \,\text{ if }\xi(x)=0.
\end{equation}
Concerning condition \eqref{eq:c1}, assume $\eta\le\xi$. If $x\not\in N_0^0$, then $R_1^x=\emptyset$. Therefore, the condition is formulated only for $x\in N_0^0$. In such case  $R_1^x=\{\sigma_{xy}^{+1,1},y\in N_1^1\}$ and the condition can be expressed as 
\begin{equation}\label{210gls}\sum_{y\in F}\Gamma_\eta(y,x)\le\sum_{y\in  F^*}\Gamma_\xi(y,x),\quad
\forall\  F\subseteq N_1^1,
\end{equation}
where $F^*\subseteq\S$ is such that $y\in F^*$ if and only if $\sigma_{xy}^{+1,1}\in\left\{\sigma_{xz}^{+1,1}:z\in F\right\}^\uparrow$.

 For  $F\subseteq N_1^1$, it is easy to check that $F^*=F\cup N_0^1$ and so, \eqref{210gls} is equivalent to 
\begin{equation*}\sum_{y\in F}\Gamma_\eta(y,x)\le\sum_{y\in F}\Gamma_\xi(y,x)+\sum_{y\in N_0^1}\Gamma_\xi(y,x), \quad \fora  F\subseteq N_1^1,
\end{equation*}
which is clearly equivalent to 
\begin{equation}\label{210sim}\sum_{y\in F}\left(\Gamma_\eta(y,x)-\Gamma_\xi(y,x)\right)\le\sum_{y\in N_0^1}\Gamma_\xi(y,x), \quad \fora  F\subseteq N_1^1.
\end{equation}
As the RHS of the inequality above is independent of $F$, condition \eqref{210gls} is satisfied if and only if \eqref{210sim} holds for the least favourable case, namely, when $F=\{y\in N_1^1: \Gamma_\eta(y,x)>\Gamma_\xi(y,x)\}$, which yields \eqref{210gs}. The proof of the equivalence between \eqref{eq:c2} and \eqref{eq:gs2} is analogous and is omitted for brevity.

\subsection{Conservative migrations without births and deaths}\label{exGS}

In the IPS investigated in \cite{GS}, the only allowed changes are conservative migrations of individuals. In this model  we have $\S=\ZZ^d$ and $W=\{0,1,\ldots,N\}, n\in\NN$. Given $v,w\in \S, k\in\NN,$ and $\eta\in\Omega=W^\S$, the migration rate of $k$ individuals from $v$ to $w$ is $\Gamma_{\eta(v),\eta(w)}^k(w-v)$, provided the resulting configuration remains in $\Omega$.

Necessary and sufficient conditions for attractiveness are found in \cite{GS} ((2.13) and (2.14) of Theorem 2.9); namely, $\fora v,w\in \S$, $\alpha,\beta,\gamma,\delta\in W$, with $\alpha\le\gamma$, $\beta\le\delta$,
\begin{equation}
	\label{GSC1}
\sum_{k'>\delta-\beta+l}\Gamma_{\alpha\beta}^{k'}(w-v)\le
\sum_{l'> l}\Gamma_{\gamma\delta}^{l'}(w-v), \fora l>0,
\end{equation}
\begin{equation}\label{GSC2}
\quad\sum_{k'> k}\Gamma_{\alpha\beta}^{k'}(w-v)\ge
\sum_{l'> \gamma-\alpha+k}\Gamma_{\gamma\delta}^{l'}(w-v), \fora k>0.
\end{equation}
Also an explicit OMC between two copies of the process is found. Note that \eqref{GSC1} and \eqref{GSC2} are conditions for attractiveness, that is, for the stochastic ordering of two processes with the same rates. Our results are more general since we also include the situation where the processes have different rates.

For the sake of illustration, our goal here is to constructively demonstrate that our conditions \eqref{eq:c1} and \eqref{eq:c2} reduce to \eqref{GSC1} and \eqref{GSC2} for attractiveness in this model. We begin showing that \eqref{eq:c1} implies \eqref{GSC1} and, to that end, we consider the configurations

\begin{equation}
	\label{eq:config}
	\begin{array}{ccccc}
		\eta(z) = \begin{cases}
			\alpha & \text{if } z = v, \\
			\beta  & \text{if } z = w,\\
			0 & \text{otherwise}
		\end{cases}
		&
		{}
		&  {}
		&  \text{and}
		& \quad
		\xi(z) = \begin{cases}
			\gamma & \text{if } z = v, \\
			\delta  & \text{if } z = w, \\
			0 & \text{otherwise}.
		\end{cases}
	\end{array}
\end{equation}
We write \eqref{eq:c1} for $\eta,\xi$ defined in \eqref{eq:config}, taking $x=w, y=v$ and  $D_1=\{\sigma_{xy}^{+k',k'}:k'>\delta-\beta+l\}$, with $l\ge0$. 
We claim that $D_1^\uparrow\subseteq\{\sigma_{xy}^{+l',l'}; l'>l\}$ and to prove this assertion, first note that, for $z\ne y, l'>0$,  $\sigma_{xz}^{+l',l'}\not\in D_1^\uparrow$, because $\xi(z)=0$. 
Also, $b=\sigma_{zx}^{+l',l'}\not \in D_1^\uparrow$ because $\eta_a(x)\not\le \xi_b(x)$ for any $a\in D_1$, since  $\xi_b(x)=\xi(x)-l'<\delta$ and $\eta_a(x)>\delta$.  
For the same reason, $\sigma_{xy}^{-l',l'}\not\in D_1^\uparrow$, so the only possible changes in $D_1^\uparrow$ are of type $\sigma_{xy}^{+l',l'}, l'>0$. 
The case $l'\le l$ is ruled out since $\xi_b(x)=\delta+l'$, while the minimum value of $\eta_a(x)$, for $a\in D_1$, is $\delta+l+1$, so $\eta_a\not\le\xi_b$, if $l'\le l$, and this completes the proof of our claim.
Therefore, from
\eqref{eq:c1} we get
\begin{equation*}
	\sum_{k'>\delta-\beta+l}\Gamma_{\alpha\beta}^{k'}(w-v)\le
	\sum_{l'> l,\, \sigma_{xy}^{+l',l'}\in D_1^\uparrow}\Gamma_{\gamma\delta}^{l'}(w-v)
	\le \sum_{l'> l}\Gamma_{\gamma\delta}^{l'}(w-v),
\end{equation*}
so \eqref{eq:c1} implies \eqref{GSC1}. An analogous argument shows that \eqref{eq:c2} implies \eqref{GSC2}.

We now prove the converse. In that aim, we must show that  $\fora \eta,\xi\in\Omega$, with $\eta\le\xi$, and $\fora x\in \S$ and $D_1\subseteq R_1^{x}$, conditions \eqref{GSC1}, \eqref{GSC2} imply \eqref{eq:c1}.

Fix $\eta\le\xi$, $x\in \S$ and let $D_1\subseteq R_1^{x}$. Note that $D_1$  can be decomposed as $D_1=\bigcup_{y\sim x}D_{1y}$, where the elements of $D_{1y}$ are of type $\sigma_{xy}^{+k,k}$ (migrations from $y$ to $x$). Now, since $D_{1y}\cap D_{1z}=\emptyset$ for $y\ne z$, for \eqref{eq:c1} it is enough to show that
\begin{equation}
	\label{GSC3}
	\sum_{a\in D_{1y}}c_a(\eta)\le\sum_{b\in\tilde{D}_{1y}^\uparrow}c_b(\xi),\quad \fora y\sim x,
	\end{equation}
where $\tilde{D}_{1y}^\uparrow\subseteq D_{1y}^\uparrow$ are the changes of $D_{1y}^\uparrow$ involving sites $x,y$. Indeed, since the sets $\tilde{D}_{1y}^\uparrow, y\sim x$, are also disjoint, it suffices to add the left-hand and right-hand sides of \eqref{GSC3} to get  \eqref{eq:c1}, as follows:
\begin{equation*}
\sum_{a\in D_1}c_a(\eta)=\sum_{y\sim x}\sum_{a\in D_{1y}}c_a(\eta) \le\sum_{y\sim x}\sum_{b\in\tilde{D}_{1y}^\uparrow}c_b(\xi)\le\sum_{b\in D_1^\uparrow}c_b(\xi).
\end{equation*}
Note that $D_{1y}$ contains changes $\sigma_{xy}^{+k,k}$, for some $k>\delta-\beta$, while $\tilde{D}_{1y}^\uparrow$ has changes $\sigma_{xy}^{+l,l}$, such that $\beta+k\le\delta+l$ and $\alpha-k\le\gamma-l$, for some $k$ such that  $\sigma_{xy}^{+k,k}\in D_{1y}$. It is not immediate that \eqref{GSC1}, which can be seen as \eqref{GSC3} for a particular class of sets $D_{1y}$,  can be extended to any choice of $D_{1y}$ and, in general, this is not true.  We need to use \eqref{GSC2} jointly with \eqref{GSC1} to prove \eqref{GSC3}, and we do so by studying the different forms that $D_{1y}$ can have.  Start by taking $D_{1y}=\{\sigma_{xy}^{+r,r}\}$, with $r=\delta-\beta+l+1$, for some $l\ge0$. 
Then  $\tilde{D}_{1y}^\uparrow$ contains changes $b=\sigma_{xy}^{+l',l'}$, such that $\eta_a\le\xi_b$;  that is, $\beta+r\le\delta+l'$ and $\alpha-r\le\gamma-l'$.
 Thus, condition \eqref{GSC3}  becomes (for ease, we omit the argument $x-y$ from $\Gamma_{\alpha\beta}^k(x-y)$):

 \begin{equation}
 	\label{GSC6}
 	\Gamma^r_{\alpha\beta}\le\sum_{l'=l+1}^{\gamma-\alpha+r}\Gamma_{\gamma\delta}^{l'}.
 	\end{equation}
 By \eqref{GSC1} and \eqref{GSC2} we have
 \begin{equation}
 	\label{eq:two ineq}
 	\sum_{k'\ge r}\Gamma^{k'}_{\alpha\beta}\le\sum_{l'> l}
 	\Gamma_{\gamma\delta}^{l'}\qquad\text{and}\qquad\sum_{k'> r}\Gamma_{\alpha\beta}^{k'}\ge\sum_{l'>\gamma-\alpha+r}
 	\Gamma_{\gamma\delta}^{l'}.
 	\end{equation}
 Therefore, 
 \begin{equation*}\Gamma_{\alpha\beta}^r+\sum_{k'\ge r+1}\Gamma_{\alpha\beta}^{k'}
 	=\sum_{k'\ge r}\Gamma_{\alpha\beta}^{k'}
 	\le\sum_{l'=l+1}^{\gamma-\alpha+r}\Gamma^{l'}_{\gamma\delta}+\sum_{l'\ge\gamma-\alpha+r+1}\Gamma_{\gamma\delta}^{l'}\le
 	\sum_{l'=l+1}^{\gamma-\alpha+r}\Gamma_{\gamma\delta}^{l'}+\sum_{k'\ge r+1}\Gamma^{k'}_{\alpha\beta},
 \end{equation*}
where we have used the first (resp. second) inequality of  \eqref{eq:two ineq} for the first (resp. second) inequality. Thus, \eqref{GSC6} follows.

Let now $D_1^y=\{\sigma_{xy}^{+r,r},\sigma_{xy}^{+r+1,r+1}\}$, with $r=\delta-\beta+l+1$, for some $l\ge0$. Condition \eqref{GSC3}  is
\begin{equation}
	\label{GSC7}
	\Gamma_{\alpha\beta}^r+\Gamma_{\alpha\beta}^{r+1}\le
	\sum_{l'=l+1}^{\gamma-\alpha+r+1}\Gamma_{\gamma\delta}^{l'}.
	\end{equation}
Reasoning as above
\begin{equation*}\Gamma_{\alpha\beta}^r+\Gamma_{\alpha\beta}^{r+1}
	+\sum_{l'\ge r+2}\Gamma_{\alpha\beta}^{l'}\le
	\sum_{l'=l+1}^{\gamma-\alpha+r+1}\Gamma_{\gamma\delta}^{l'}
	+\sum_{l'\ge\gamma-\alpha+r+2}\Gamma_{\gamma\delta}^{l'}
	\le\sum_{l'=l+1}^{\gamma-\delta+r+1}\Gamma_{\gamma\delta}^{l'}+\sum_{l'\ge r+2}\Gamma_{\alpha\beta}^{l'}\end{equation*}
and so, \eqref{GSC7} holds. This argument works also for $D_1^y$ of the form $\{\sigma_{xy}^{+r,r}: r=\delta-\beta+l+1,l\in[m,n]\}$. The situation where $D_1^y$ has gaps is implied by the cases above. Take, for instance, $D_1^y=\{\sigma_{xy}^{+r,r},\sigma_{xy}^{+r+2,r+2}\}$ with $r=\delta-\beta+l+1$, for some $l\ge0$. Then
 	\begin{equation*}\Gamma_{\alpha\beta}^r+\Gamma_{\alpha\beta}^{r+2}\le
 		\Gamma_{\alpha\beta}^r+\Gamma_{\alpha\beta}^{r+1}+\Gamma_{\alpha\beta}^{r+2}
 		\le\sum_{l'=l+1}^{\gamma-\alpha+r+2}\Gamma_{\gamma\delta}^{l'}.\end{equation*}
 	 	So, we have verified that \eqref{GSC1} and \eqref{GSC2} together imply \eqref{eq:c1}. Reasoning analogously we see that they also imply \eqref{eq:c2}.

\subsection{Examples of construction of the coupling}\label{excp}
In this section we illustrate the construction of the OMC in Definition \ref{def:gencoup}. We analyse three examples: the first two involve couplings of two copies of attractive general exclusion processes and  two copies of  attractive one-dimensional two-species exclusion models, as detailed in  Sections \ref{exGep} and \ref{exGS} respectively. The third example presents a coupling between two distinct processes, one  allowing non-conservative migrations. While OMCs for the first two examples have already been provided in \cite{GS23} and \cite{GS} respectively, our aim here is to demonstrate the workings of our general construction and highlight the differences between our approach and the more ``ad hoc'' solutions presented in those papers.

\subsubsection{General exclusion processes}\label{exGepcoup}

In Propositions 3.2 and 3.3 of \cite{GS23}, an OMC between two copies of an attractive general exclusion process is provided. Here, we use the notation in Section \ref{exGep} and present the construction of the coupling in Definition \ref{def:gencoup} for this class of processes. 

Let $\eta\le\xi$. Initially, for each $x\in\S$, we need to formulate and solve problem $P^x$. The subsequent analysis distinguishes three cases, depending on which set $x$ belongs to:

\begin{enumerate}
\item If $x\in N_0^1$, then $T_1^x=T_2^x=\emptyset$, rendering the problem $P^x$ empty in this case.
\item If $x\in N_0^0$, we have $R_1^x=\{\sigma_{xy}^{+1,1}:y\in N_1^1\}$, $S_1^x=\emptyset$,  $R_2^x=\emptyset$, $S_2^x=\{\sigma_{xy}^{+1,1}:y\in N_0^1\cup N_1^1\}$ and so, $T_1^x=R_1^x$, $T_2^x=S_2^x$.
To determine the arcs from nodes $b\in T_2^x$ to nodes $a\in T_1^x$, it is necessary to check if $\eta_a\le\xi_b$. 

First note that if $a=\sigma_{xy}^{+1,1}, b=\sigma_{xz}^{+1,1}$, with $y\in N_1^1, z\in N_0^1$ (obviously $y\ne z$), then $\eta_a\le\xi_b$. Consequently, an arc exists from every $b$ to every  $a$, in this situation.

Now, if $a=\sigma_{xy}^{+1,1}, b=\sigma_{xz}^{+1,1}$, with $y,z \in N_1^1$, it is straightforward to check that $\eta_a\le\xi_b$ if and only if $y=z$.

  The lower and upper bounds of $P^x$ are presented in Table \ref{tab:Gepx}. There, the second line represents the existence of an arc from every node in $\{\sigma_{xz}^{+1,1}:z\in N_0^1\}$  to every node in $\{\sigma_{xy}^{+1,1}:y\in N_1^1\}$, while the third line indicates the presence of a unique arc from each node of $\{\sigma_{xy}^{+1,1}:y\in N_1^1\}$ to the node $\sigma_{xy}^{+1,1}$ (sharing the same $y$). 
  
\begin{table}[!h]
	\begin{center}
				{\caption{\label{tab:Gepx}Lower ($l$) and upper ($u$) bounds for arcs $(o,d)$ in problem $P^x$.}}
		{\scalebox{1}{
				\begin{tabular}{ccccc}
					\hline
					$o$& & $d$     & $l(o,d)$ & $u(o,d)$  \\
					\hline
					$O$& & $\sigma_{xy}^{+1,1},y\in N_0^1\cup N_1^1$     & 0  & $\Gamma_\xi(y,x)$   \\
					$\sigma_{xz}^{+1,1},z\in N_0^1$& &$\sigma_{xy}^{+1,1},y\in N_1^1$ &0&$\infty$\\
					$\sigma_{xy}^{+1,1},y\in N_1^1$& &$\sigma_{xy}^{+1,1},y\in N_1^1$ &0&$\infty$\\
					$\sigma_{xy}^{+1,1},y\in N_1^1$& & $Z$     & $\Gamma_\eta(y,x)$  & $\Gamma_\eta(y,x)$   \\
					\hline
		\end{tabular}}}
	\end{center}
\end{table}
	
To find a feasible solution for $P^x$ we start by letting $f^x(\sigma_{xy}^{+1,1},\sigma_{xy}^{+1,1})=\Gamma_\eta(y,x)\wedge\Gamma_\xi(y,x)$. Subsequently, the nodes in $\{\sigma_{xy}^{+1,1}:y\in N_1^1\}\subseteq T_2^x$ can be removed, as they are unable to send further flow to any node in $\{\sigma_{xy}^{+1,1}:y\in N_1^1\}\subseteq T_1^x$. Thus, we are left with a transportation problem from the origins $\{\sigma_{xy}^{+1,1}:y\in N_0^1\}\subseteq T_2^x$, with offers $\Gamma_\xi(y,x)$, to the destinations $\{\sigma_{xy}^{+1,1}:y\in N_1^1\}= R_1^x$, with demands $\left(\Gamma_\eta(y,x)-\Gamma_\xi(y,x)\right)^+$, the feasibility of which is guaranteed by \eqref{eq:c1} (or \eqref{eq:gs1}). A particular solution can be readily determined by introducing a dummy destination with a demand given by
\begin{equation*}\sum_{y\in N_0^1}\Gamma_\xi(y,x)-\sum_{y\in N_1^1}\left(\Gamma_\eta(y,x)-\Gamma_\xi(y,x)\right)^+.
\end{equation*}
Subsequently, any algorithm designed to find an initial solution for transportation problems, such as the Northwest Corner Rule, can be applied. It is worth noting that costs may be assigned to the arcs so that the optimal solution for the transportation problem favours specific coupled changes; refer to Remark \ref{objetivo}.

\item For brevity, we omit the case $x\in N_1^1$, as it is analogous to the case $x\in N_0^0$.
\end{enumerate}

After solving problem $P^x$ for each $x\in\S$, we must solve problems $P^{xy}$, which depend on the sites $x,y\in\S$. A straightforward check reveals that either $T^{xy}_1=\emptyset$ or $T_2^{xy}=\emptyset$ for cases when $x,y\in N_0^0$, $x,y\in N_1^1$, $x\in N_0^1$, or $y\in N_0^1$. Consequently, the only instances of problems $P^{xy}$ that require resolution are those where one of $x$ or $y$ belongs to $N_0^0$ and the other to $N_1^1$. 

Take $x\in N_0^0$ and $y\in N_1^1$. The nodes are defined through the following sets (refer to Definitions \ref{def:goodbad} and \ref{def:Txy}): $G_1^{-x}=G_1^{-y}=G_1^{-x\bullet}=B_1^{-xy}=\emptyset$, $G_1^{-y\bullet}=\{\sigma_{yz}^{-1,1}:z\in N_0^1\}$, $B_1^{+xy}=\{\sigma_{xy}^{+1,1}\}$ and $G_2^{+x}=G_2^{+y}=G_2^{+y\bullet}=B_2^{-xy}=\emptyset$, $G_2^{+x\bullet}=\{\sigma_{xz}^{+1,1}:z\in N_0^1\}$, $B_2^{+xy}=\{\sigma_{xy}^{+1,1}\}$. The arcs, along with their lower and upper bounds, are detailed in Table \ref{tab:Gepxcoup}. 
\begin{table}[!h]
	\begin{center}
				{\caption{\label{tab:Gepxcoup}Lower ($l$) and upper ($u$) bounds for arcs $(o,d)$ in problem $P^{xy}$.}}
		{\scalebox{1}{
				\begin{tabular}{ccccc}
					\hline
					$o$& & $d$     & $l(o,d)$ & $u(o,d)$  \\
					\hline
					$O$& & $\sigma_{xy}^{+1,1}$     & $\Gamma_\xi(y,x)$  & $\Gamma_\xi(y,x)$   \\
					$O$& & $\sigma_{xz}^{+1,1},z\in N_0^1$     & 0  &$f^x(\sigma_{xz}^{+1,1},\sigma_{xy}^{+1,1})$  \\
					$\sigma_{xy}^{+1,1}$& &$\sigma_{yz}^{-1,1},z\in N_0^1$ &0&$\infty$\\
					$\sigma_{xy}^{+1,1}$& &$\sigma_{xy}^{+1,1}$ &0&$\infty$\\
					$\sigma_{xz}^{+1,1},z\in N_0^1$& &$\sigma_{xy}^{+1,1}$ &0&$\infty$\\
					$\sigma_{yz}^{-1,1}, z\in N_0^1$& & $Z$     & 0  & $f^y(\sigma_{xy}^{+1,1},\sigma_{yz}^{-1,1}$)   \\
					$\sigma_{xy}^{+1,1}$& & $Z$     & $\Gamma_\eta(y,x)$  & $\Gamma_\eta(y,x)$   \\
					\hline
		\end{tabular}}}
	\end{center}
\end{table}

A solution for this problem is (only the flows from $b\in T_2^{xy}$ to $a\in T_1^{xy}$ are shown):
\begin{equation*}f^{xy}(b,a)=
\begin{cases}
\Gamma_\eta(y,x)\wedge\Gamma_\xi(y,x)&{\rm if\ } a=\sigma_{xy}^{+1,1},b=\sigma_{xy}^{+1,1},\\
f^y(\sigma_{xy}^{+1,1},\sigma_{yz}^{-1,1})&{\rm if\ } a=\sigma_{yz}^{-1,1},z\in N_0^1,b=\sigma_{xy}^{+1,1},\\
f^x(\sigma_{xz}^{+1,1},\sigma_{xy}^{+1,1})&{\rm if\ } a=\sigma_{xy}^{+1,1}, b=\sigma_{xz}^{+1,1},z\in N_0^1. \end{cases}
\end{equation*}

Then, following Definition \ref{def:gencoup}, the generator of the OMC, acting on  $g:\{(\eta,\xi)\in\Omega\times \Omega:\ \eta\le\xi\}\to\mathbb{R}$,  is
	\begin{alignat}{1}
	\label{gs23coup1}
	 {\cal L}_cg(\eta,\xi)=&\sum_{x\in N_1^1}\sum_{z\in N_0^1}
	 \left(\Gamma_\eta(x,z)-\sum_{y\in N_0^0}f^y(\sigma_{yx}^{+1,1},\sigma_{xz}^{-1,1})\right)\left(g\left(\eta_a,\xi\right)-g(\eta,\xi)\right)\\
	\label{gs23coup2}
	 &+\sum_{x\in N_0^0}\sum_{z\in N_0^1}\left(\Gamma_\xi(z,x)-\sum_{y\in N_1^1}f^x(\sigma_{xz}^{+1,1},\sigma_{xy}^{+1,1})\right)\left(g\left(\eta,\xi_b\right)-g(\eta,\xi)\right)\\
	\label{gs23coup3}
	 &+\sum_{x\in N_0^0}\sum_{y\in N_1^1}\sum_{z\in N_0^1}f^y(\sigma_{xy}^{+1,1},\sigma_{yz}^{-1,1})\left(g\left(\eta_a,\xi_b\right)-g(\eta,\xi)\right)\\
	\label{gs23coup4}
	 &+\sum_{x\in N_0^0}\sum_{y\in N_1^1}\sum_{z\in N_0^1}f^x(\sigma_{xz}^{+1,1},\sigma_{xy}^{+1,1})\left(g\left(\eta_a,\xi_b\right)-g(\eta,\xi)\right)\\
	\label{gs23coup5}
	 &+\sum_{x\in N_0^0}\sum_{y\in N_1^1}\left(\Gamma_\eta(y,x)\wedge\Gamma_\xi(y,x)\right)\left(g\left(\eta_a,\xi_b\right)-g(\eta,\xi)\right),
	\end{alignat}
where $a$ in \eqref{gs23coup1} stands for $\sigma_{xz}^{-1,1}$, $b$ in \eqref{gs23coup2} stands for $\sigma_{xz}^{+1,1}$ and $(a,b)$ in \eqref{gs23coup3}, \eqref{gs23coup4} and \eqref{gs23coup5} stand for
$(\sigma_{yz}^{-1,1},\sigma_{xy}^{+1,1})$, $(\sigma_{xy}^{+1,1},\sigma_{xz}^{+1,1})$ and $(\sigma_{xy}^{+1,1},\sigma_{xy}^{+1,1})$,  respectively.

\subsubsection{One-dimensional two-species exclusion model}\label{exGScoup}
In \cite{GS} the authors give the explicit expression of an OMC between two copies of an attractive  conservative migration process, belonging to the class analysed in Section \ref{exGS}. The formulation of the coupling in their equations (2.16)-(2.19) demonstrates ingenuity and is far from trivial. Indeed, its application is challenging, in particular because behind the compact writing there are many subtleties that emerge when specific rates are set.  For instance, in their Section 4.3   they study the one-dimensional, two-species exclusion model, give necessary and sufficient conditions for attractiveness and present the explicit expression of an OMC, which is fully written in their Table 1. In this section we show how our construction of the coupling works for this model. 

The model has $W=\{-1,0,1\}$, $S=\mathbb{Z}$ and $x\sim y$ if and only if $\vert y-x\vert=1$. There are ten different rates $\Gamma_{\alpha\beta}^k(z)$. In what follows, for simplicity, we assume that $\Gamma_{\alpha\beta}^k(+1)=\Gamma_{\alpha\beta}^k(-1)$. Also, to keep our convention  that $W$ is a subset of $\NN\cup\{0\}$, we add 1 to the values in that example; i.e., their $-1$ becomes 0 and so on. To simplify the notation, we  write $r_1$ for $\Gamma_{20}^2$, $r_2$ for $\Gamma_{11}^1$, $r_3$ for $\Gamma_{10}^1$, $r_4$ for $\Gamma_{21}^1$ and $r_5$ for $\Gamma_{20}^1$. The changes and their rates are summarized in Table \ref{tab:excrates}.
\begin{table}[!h]
	\begin{center}
				{\caption{\label{tab:excrates}Rates in the symmetric one-dimensional two-species exclusion model.}}
		{\scalebox{1}{
				\begin{tabular}{cccc}
					\hline
					Change&  Rate     & Change & Rate  \\
					\hline
					$(1\ 0)\rightarrow (0\ 1)$&$r_3$&					$(0\ 1)\rightarrow (1\ 0)$&$r_3$\\		
			$(2\ 0)\rightarrow (0\ 2)$&$r_1$&					$(0\ 2)\rightarrow (2\ 0)$&$r_1$\\	
				$(2\ 0)\rightarrow (1\ 1)$&$r_5$&					$(0\ 2)\rightarrow (1\ 1)$&$r_5$\\		
			$(1\ 1)\rightarrow (0\ 2)$&$r_2$&					$(1\ 1)\rightarrow (2\ 0)$&$r_2$\\			
		$(2\ 1)\rightarrow (1\ 2)$&$r_4$&					$(1\ 2)\rightarrow (2\ 1)$&$r_4$\\
					\hline
		\end{tabular}}}
	\end{center}
\end{table}

The necessary and sufficient conditions for attractiveness, presented in  Proposition 4.3 of \cite{GS}, are given by 
\begin{equation}
\label{condejGS}
r_1\vee r_2\le r_3\wedge r_4\le r_3\vee r_4\le r_1+r_5.
\end{equation}

In order to give the coupling we need to compute the rates of changes involving every site $v\in \S$, from every configuration $(\eta,\xi)$. We fix $\eta\le\xi, v\in \S$ and give the explicit rates of the coupling that modify the value at site $v$. Looking at expressions
\eqref{eq:gen1}-\eqref{eq:genx} we see that the only values of $x\in \S$ which can affect the rate of site $v$ are $x=v-1,v,v+1$. Thus, the values of $f^{v-1}$, $f^v$, $f^{v+1}$ are needed. Further, the values of $(x,y)$ in \eqref{eq:gen4}, \eqref{eq:gen6} and \eqref{eq:genx}, involving site $v$, are $\{v-1,v\}$ and $\{v,v+1\}$, so $f^{v-1, v}$ and $f^{v, v+1}$ are needed. This means that we have to solve the problems $P^{v-1}$, $P^v$, $P^{v+1}$, $P^{v-1,v}$ and $P^{v,v+1}$.  Since the rate of change of a site $x$ depends on $x-1$,  $x$ and $x+1$, these problems are determined by the values of $\eta$ and $\xi$ at sites $v-2,v-1,v,v+1,v+2$. 

We consider the particular case of $(0\ 2\ 0\ 1\ 2)$ for $\eta$ and $(1\ 2\ 1\ 1\ 2)$ for $\xi$. 
In order to have the general expression of the coupling we should repeat the procedure below, for all the values of the 5-tuples, for $\eta$ and $\xi$, with $\eta\le\xi$. 

We now state the problems $P^{v-1}$, $P^v$ and $P^{v+1}$. To ease the exposition we represent the nodes as the result of the changes rather than the changes themselves. For instance, the change $\sigma_{v-1,v}^{-2,2}$ in the first process, that is, the migration of two particles from $v-1$ to $v$, is represented by $(0\ 0\ 2\ 1\ 2)$, which is the result of that change from $(0\ 2\ 0\ 1\ 2)$. The nodes, arcs and bounds of problems $P^{v-1}$, $P^v$ and $P^{v+1}$ are shown in Table \ref{probx-1xx+1}. Note that the problems are feasible by \eqref{condejGS}. A feasible solution to each problem, satisfying,  condition (f) of Lemma \ref{lem:f1xtilde}, is given in Table \ref{solp1}.

\begin{table}[!h]
	\begin{center}
				{\caption{\label{probx-1xx+1}Lower ($l$) and upper ($u$) bounds for arcs $(o,d)$ in problems $P^{v-1}$, $P^v$ and $P^{v+1}$.}}
		{\scalebox{1}{
				\begin{tabular}[t]{cccc}
					\hline
					& Problem $P^{v-1}$&&\\
					\hline
					$o$&  $d$     & $l$ & $u$  \\
					\hline
					$O$& $(2\ 1\ 1\ 1\ 2)$     &$ r_4$  & $r_4$   \\				
					$O$&  $(1\ 1\ 2\ 1\ 2)$     &$ r_4$  & $r_4$   \\
					$(2\ 1\ 1\ 1\ 2)$ & $(1\ 1\ 0\ 1\ 2)$     & $0$  & $\infty$   \\
					$(2\ 1\ 1\ 1\ 2)$ & $(2\ 0\ 0\ 1\ 2)$     & $0$  & $\infty$   \\
					$(2\ 1\ 1\ 1\ 2)$ & $(0\ 1\ 1\ 1\ 2)$     & $0$  & $\infty$   \\
					$(1\ 1\ 2\ 1\ 2)$ & $(1\ 1\ 0\ 1\ 2)$     & $0$  & $\infty$   \\
					$(1\ 1\ 2\ 1\ 2)$ & $(0\ 1\ 1\ 1\ 2)$     & $0$  & $\infty$   \\
					$(1\ 1\ 2\ 1\ 2)$ & $(0\ 0\ 2\ 1\ 2)$     & $0$  & $\infty$   \\
					$(1\ 1\ 0\ 1\ 2)$ & $Z$    & $0$  & $r_5$   \\
					$(2\ 0\ 0\ 1\ 2)$ & $Z$    & $0$  & $r_1$   \\
					$(0\ 1\ 1\ 1\ 2)$ & $Z$    & $0$  & $r_5$   \\
					$(0\ 0\ 2\ 1\ 2)$ & $Z$    & $0$  & $r_1$   \\
					\hline
		\end{tabular}

		\hfill
					\begin{tabular}[t]{cccc}
					\hline
					&Problem $P^v$&&\\
					\hline
					$o$&  $d$     & $l$ & $u$  \\
					\hline
					$O$& $(1\ 1\ 2\ 1\ 2)$     &$ 0$  & $r_4$   \\				
					$O$&  $(1\ 2\ 2\ 0\ 2)$     &$ 0$  & $r_2$   \\
					$(1\ 1\ 2\ 1\ 2)$  & $(0\ 0\ 2\ 1\ 2)$     & $0$  & $\infty$   \\
					$(0\ 0\ 2\ 1\ 2)$ & $Z$    & $r_1$  & $r_1$   \\
									\hline
\end{tabular}		
\hfill
					\begin{tabular}[t]{cccc}
					\hline
					&Problem $P^{v+1}$&&\\
					\hline
					$o$&  $d$     & $l$ & $u$  \\
					\hline
					$O$& $(1\ 2\ 2\ 0\ 2)$     &$ r_2$  & $r_2$   \\				
					$O$&  $(1\ 2\ 0\ 2\ 2)$     &$ 0$  & $r_2$   \\
					$O$&  $(1\ 2\ 1\ 2\ 1)$     &$ 0$  & $r_4$   \\
					$(1\ 2\ 2\ 0\ 2)$ & $(0\ 2\ 1\ 0\ 2)$     & $0$  & $\infty$   \\
					$(1\ 2\ 0\ 2\ 2)$ & $(0\ 2\ 0\ 2\ 1)$     & $0$  & $\infty$   \\
					$(1\ 2\ 1\ 2\ 1)$ & $(0\ 2\ 0\ 2\ 1)$     & $0$  & $\infty$   \\
					$(0\ 2\ 0\ 2\ 1)$ & $Z$    & $r_4$  & $r_4$   \\
					$(0\ 2\ 1\ 0\ 2)$ & $Z$    & $0$  & $r_3$   \\
					\hline
		\end{tabular}
		}}
	\end{center}
\end{table}

\begin{table}[!h]
	\begin{center}
				{\caption{\label{solp1}Solutions to problems $P^{v-1}$, $P^v$ and $P^{v+1}$. Only nonzero flows are shown.}}
		{\scalebox{1}{
				\begin{tabular}[t]{ccc}
					\hline
					& $P^{v-1}$&\\
					\hline
					$o$&  $d$     & $f^{v-1}$  \\
					\hline
					$O$& $(2\ 1\ 1\ 1\ 2)$     &$ r_4$    \\				
					$O$&  $(1\ 1\ 2\ 1\ 2)$     &$ r_4$   \\
					$(2\ 1\ 1\ 1\ 2)$ & $(1\ 1\ 0\ 1\ 2)$     & $r_4-r_1$   \\
					$(2\ 1\ 1\ 1\ 2)$ & $(2\ 0\ 0\ 1\ 2)$     & $r_1$   \\
					$(1\ 1\ 2\ 1\ 2)$ & $(0\ 1\ 1\ 1\ 2)$     & $r_4-r_1$   \\
					$(1\ 1\ 2\ 1\ 2)$ & $(0\ 0\ 2\ 1\ 2)$     & $r_1$   \\
					$(1\ 1\ 0\ 1\ 2)$ & $Z$    & $r_4-r_1$   \\
					$(2\ 0\ 0\ 1\ 2)$ & $Z$    & $r_1$  \\
					$(0\ 1\ 1\ 1\ 2)$ & $Z$    & $r_4-r_1$   \\
					$(0\ 0\ 2\ 1\ 2)$ & $Z$    & $r_1$   \\
					\hline
		\end{tabular}
		\hspace{0.5cm}
					\begin{tabular}[t]{ccc}
					\hline
					&Problem $P^v$&\\
					\hline
					$o$&  $d$     & $f^v$  \\
					\hline
					$O$& $(1\ 1\ 2\ 1\ 2)$     &$r_1$   \\				
					$(1\ 1\ 2\ 1\ 2)$  & $(0\ 0\ 2\ 1\ 2)$     & $r_1$   \\
					$(0\ 0\ 2\ 1\ 2)$ & $Z$    & $r_1$   \\
									\hline
\end{tabular}		
\hspace{0.5cm}
					\begin{tabular}[t]{ccc}
					\hline
					&Problem $P^{v+1}$&\\
					\hline
					$o$&  $d$     & $f^{v+1}$  \\
					\hline
					$O$& $(1\ 2\ 2\ 0\ 2)$     &$r_2$   \\				
					$O$&  $(1\ 2\ 1\ 2\ 1)$     &$r_4$   \\
					$(1\ 2\ 2\ 0\ 2)$ & $(0\ 2\ 1\ 0\ 2)$     & $r_2$   \\
					$(1\ 2\ 1\ 2\ 1)$ & $(0\ 2\ 0\ 2\ 1)$     & $r_4$   \\
					$(0\ 2\ 0\ 2\ 1)$ & $Z$    & $r_4$   \\
					$(0\ 2\ 1\ 0\ 2)$ & $Z$    & $r_2$   \\
					\hline
		\end{tabular}
		}}
	\end{center}
\end{table}

Once we have the solutions to the individual problems, we can write problems $P^{v-1,v}$ and $P^{v,v+1}$ following Section \ref{secprob2}. These problems and their solutions are shown in Tables \ref{probpares} and \ref{solpares} respectively. 

\begin{table}[!h]
	\begin{center}
				{\caption{\label{probpares}Lower ($l$) and upper ($u$) bounds for arcs $(o,d)$ in problems $P^{v-1,v}$ and $P^{v,v+1}$.}}
		{\scalebox{1}{
				\begin{tabular}[t]{cccc}
					\hline
					& $P^{v-1,v}$&&\\
					\hline
					$o$&  $d$     & $l$ & $u$  \\
					\hline
					$O$& $(1\ 1\ 2\ 1\ 2)$     &$ r_4$  & $r_4$   \\				
					$(1\ 1\ 2\ 1\ 2)$     & $(0\ 0\ 2\ 1\ 2)$     & $0$  & $\infty$   \\
					$(1\ 1\ 2\ 1\ 2)$     & $(0\ 1\ 1\ 1\ 2)$     & $0$  & $\infty$   \\
					$(0\ 0\ 2\ 1\ 2)$ & $Z$    & $r_1$  & $r_1$   \\
					 $(0\ 1\ 1\ 1\ 2)$ & $Z$    & $0$  & $r_4-r_1$   \\
					\hline
		\end{tabular}
		\hspace{2cm}
					\begin{tabular}[t]{cccc}
					\hline
					&$P^{v,v+1}$&&\\
					\hline
					$o$&  $d$     & $l$ & $u$  \\
					\hline
					$O$& $(1\ 2\ 2\ 0\ 2)$     &$ r_2$  & $r_2$   \\				
					$(1\ 2\ 2\ 0\ 2)$  & $(0\ 2\ 1\ 0\ 2)$     & 0  & $\infty$   \\
					$(0\ 2\ 1\ 0\ 2)$  & $Z$    & 0  & $r_2$   \\
									\hline
\end{tabular}		
		}}
	\end{center}
\end{table}

\begin{table}[!h]
	\begin{center}
				{\caption{\label{solpares}Solutions of  problems $P^{v-1,v}$ and $P^{v,v+1}$.}}
		{\scalebox{1}{
				\begin{tabular}[t]{ccc}
					\hline
					& $P^{v-1,v}$&\\
					\hline
					$o$&  $d$     & $f^{v-1,v}$  \\
					\hline
					$O$& $(1\ 1\ 2\ 1\ 2)$     &$ r_4$   \\				
					$(1\ 1\ 2\ 1\ 2)$     & $(0\ 0\ 2\ 1\ 2)$     & $r_1$   \\
					$(1\ 1\ 2\ 1\ 2)$     & $(0\ 1\ 1\ 1\ 2)$     & $r_4-r_1$   \\
					$(0\ 0\ 2\ 1\ 2)$ & $Z$    & $r_1$   \\
					 $(0\ 1\ 1\ 1\ 2)$ & $Z$    & $r_4-r_1$   \\
					\hline
		\end{tabular}
		\hspace{2cm}
					\begin{tabular}[t]{ccc}
					\hline
					&$P^{v,v+1}$&\\
					\hline
					$o$&  $d$     & $f^{v,v+1}$  \\
					\hline
					$O$& $(1\ 2\ 2\ 0\ 2)$     &$ r_2$    \\				
					$(1\ 2\ 2\ 0\ 2)$  & $(0\ 2\ 1\ 0\ 2)$     &  $r_2$   \\
					$(0\ 2\ 1\ 0\ 2)$  & $Z$    &  $r_2$   \\
									\hline
\end{tabular}		
		}}
	\end{center}
\end{table}

The solutions $f^{v-1}$, $f^v$, $f^{v+1}$, $f^{v-1,v}$ and $f^{v,v+1}$ are then used to write the explicit rates of the coupling. The coupled changes involving site $v$ (that is, those where the two components change together), given in terms \eqref{eq:gen3} to \eqref{eq:gen6} of Section \ref{sec:coc}, are shown in Table \ref{coupexp1}. The rest of changes involving site $v$ are made independently in marginals $(\eta_t)$ and $(\xi_t)$. It can be checked that, for these particular values of $\eta$ and $\xi$ at $\{v-2,v-1,v,v+1,v+2\}$, the coupled changes in Table \ref{coupexp1} coincide with those in Table 1 of \cite{GS}.
It is worth noting that for some other values our coupling deviates from the one given in  \cite{GS}. For instance, let the values at $\{v-2,v-1,v,v+1,v+2\}$ in $\eta$ and $\xi$ be $(0\ 0\ 1\ 0\ 0)$ and $(0\ 0\ 2\ 1\ 0)$, respectively. The coupling in Table 1 of \cite{GS} assigns a rate $r_3\wedge r_4$ to the change $((0\ 0\ 0\ 1\ 0),(0\ 0\ 1\ 2\ 0))$. However, since the change $(0\ 0\ 0\ 1\ 0)$ is not in $B_1^{+x}\cup B_1^{-xy}$ for any $x\in\mathcal{S}$, $y\sim x$  and the change $(0\ 0\ 1\ 2\ 0)$ is not in $B_2^{-x}\cup B_2^{-xy}$ for any $x\in\mathcal{S}$, $y\sim x$, the coupled change does not appear in any of the terms \eqref{eq:gen3}-\eqref{eq:genx}. This means that, even though our construction gives us large freedom to define the rates (since it is likely that the network flow problems have multiple solutions), not every OMC between two processes follows the formula in Definition \ref{def:gencoup}.

\begin{table}[!h]
	\begin{center}
				{\caption{\label{coupexp1}Coupled changes involving site $v$.}}
		{\scalebox{1}{
				\begin{tabular}[t]{ccc}
					\hline
					$\eta_t$&  $\xi_t$     & Rate  \\
					\hline
					$(0\ 0\ 2\ 1\ 2)$ &$(1\ 1\ 2\ 1\ 2)$     &$ r_1$   \\				
					$(0\ 1\ 1\ 1\ 2)$     & $(1\ 1\ 2\ 1\ 2)$     & $r_4-r_1$   \\
					$(0\ 2\ 1\ 0\ 2)$     & $(1\ 2\ 2\ 0\ 2)$     & $r_2$   \\
					\hline
\end{tabular}		
		}}
	\end{center}
\end{table}

\subsubsection{An example with non-conservative migrations}\label{excoup2}
It can be argued that, for the particular example discussed in Section \ref{exGScoup}, the changes of the coupling can be found through simple inspection and the method proposed in this paper, which solves several flow problems, is intricate. Nevertheless, it is important to emphasize that our method remains applicable regardless of the model's complexity and, if necessary, it can be implemented on a computer. 
We now show how to construct an OMC in an example where births, deaths and non-conservative migrations are present. 

Consider the processes $(\eta_t)$, $(\xi_t)$, with $W=\{0,1,\ldots,M\}$, for   $M>4$, and $S=\mathbb{Z}$. Also, let  $x\sim y$ if and only if   $\vert x-y\vert=1$. The process $(\eta_t)$ has: death of one individual at rate $\mu_1$ and of two individuals at rate $\mu_2$; conservative migrations of one individual at rate $\gamma_1$ and non-conservative migrations of one individual, giving birth to another individual at rate $\gamma_2$. The process $(\xi_t)$ has: arrivals of one individual at rate $\alpha_1$ and of two individuals at rate $\alpha_2$; death of one individual at rate $\mu_1$, of two individuals at rate $\mu_2$ and conservative migrations of one individual at rate $\beta$. The rates are shown in Table \ref{tasasinv}.

\begin{table}[!h]
	\begin{center}
				{\caption{\label{tasasinv}Rates of the processes in Section \ref{excoup2}. }}
		{\scalebox{1}{
				\begin{tabular}[t]{cc}
					\hline
					 $(\eta_t)$&\\
					\hline
					Change &  Rate\\
					\hline
					$i\rightarrow i-1$     &$\mu_1$   \\				
					$i\rightarrow i-2$     &$\mu_2 $  \\				
					$(i,j)\rightarrow (i-1,j+1)$     &$\gamma_1$   \\				
					$(i,j)\rightarrow (i+1,j-1)$     &$\gamma_1 $  \\				
					$(i,j)\rightarrow (i-1,j+2)$     &$\gamma_2  $ \\				
					$(i,j)\rightarrow (i+2,j-1)$     &$\gamma_2  $ \\				
					\hline
		\end{tabular}
		\hspace{2cm}
					\begin{tabular}[t]{cc}
					\hline
					 $(\xi_t)$&\\
					\hline
					Change &  Rate\\
					\hline
					$i\rightarrow i+1$     &$\alpha_1$   \\				
					$i\rightarrow i+2$     &$\alpha_2 $  \\				
					$i\rightarrow i-1$     &$\mu_1  $ \\				
					$i\rightarrow i-2$     &$\mu_2  $ \\				
					$(i,j)\rightarrow (i-1,j+1)$     &$\beta$ \\				
					$(i,j)\rightarrow (i+1,j-1)$     &$\beta  $ \\				
					\hline
\end{tabular}		
		}}
	\end{center}
\end{table}

As in Section \ref{exGScoup}, we  find the rates of the coupled process fixing a site $v$ and the values of the configurations at $(v-2,v-1,v,v+1,v+2)$. Namely, we choose $(1\ 1\ 2\ 1\ 1)$ for the values of $\eta$ and $(1\ 2\ 3\ 1\ 1)$ for the values of $\xi$. 
We assume the following relations among the parameters: $\gamma_2\le (\alpha_1+\alpha_2)\wedge(\alpha_2/2)$, $0\le\mu_2-\mu_1\le\gamma_1$ and  $0\le\beta-\gamma_1\le\gamma_1\wedge (\gamma_2/2)$. Some of these conditions are necessary for the construction of the OMC while others are chosen for the sake of simplicity in presentation.

We  proceed as in Section \ref{exGScoup}, by defining and solving the problems $P^{v-1}$, $P^v$ and $P^{v+1}$. 
The nodes, arcs and bounds of problems $P^{v-1}$, $P^v$ and $P^{v+1}$ are presented  in Table \ref{2probx-1xx+1} where, for simplicity, we  omit the arcs from $T_2^z$ to $T_1^z$, with respective lower and upper bounds 0 and $\infty$. However, it must be noted that there exists an arc from every $b\in T_2^z$ to every $a\in T_1^z$, such that $\eta_a\le\xi_b$; see Table \ref{tab:P1xbounds}. 
  Solutions for the problems, satisfying condition (f) of Lemma \ref{lem:f1xtilde}, are given in Table \ref{2solp1}. To keep the table small, only the (strictly positive) flows between arcs from $b\in T_2^z$ to  $a\in T_1^z$ are shown, as they  determine the flows on the rest of the arcs.

\begin{table}[!h]
	\begin{center}
				{\caption{\label{2probx-1xx+1}Lower ($l$) and upper ($u$) bounds for arcs $(o,d)$ in problems $P^{v-1}$, $P^v$ and $P^{v+1}$. Arcs from $b\in T_2^z$ to  $a\in T_1^z$ are omitted. }}
		{\scalebox{1}{
				\begin{tabular}[t]{cccc}
					\hline
					& Problem $P^{v-1}$&&\\
					\hline
					$o$&  $d$     & $l$ & $u$  \\
					\hline
					$O$& $(1\ 0\ 3\ 1\ 1)$     &$ \mu_2$  & $\mu_2$   \\				
					$O$&  $(1\ 3\ 3\ 1\ 1)$     &$ 0$  & $\alpha_1$   \\
					$O$&  $(1\ 4\ 3\ 1\ 1)$     &$ 0$  & $\alpha_2$   \\
					$O$&  $(0\ 3\ 3\ 1\ 1)$     &$ 0$  & $\beta$   \\
					$O$&  $(1\ 3\ 2\ 1\ 1)$     &$ 0$  & $\beta$   \\
					$(1\ 3\ 1\ 1\ 1)$ & $Z$    & $\gamma_2$  & $\gamma_2$   \\
					$(0\ 3\ 2\ 1\ 1)$ & $Z$    & $\gamma_2$  & $\gamma_2$   \\
					$(1\ 0\ 2\ 1\ 1)$ & $Z$    & $0$  & $\mu_1$   \\
					$(2\ 0\ 2\ 1\ 1)$ & $Z$    & $0$  & $\gamma_1$   \\
					$(3\ 0\ 2\ 1\ 1)$ & $Z$    & $0$  & $\gamma_2$   \\
					$(1\ 0\ 4\ 1\ 1)$ & $Z$    & $0$  & $\gamma_2$   \\
					$(1\ 0\ 3\ 1\ 1)$ & $Z$    & $0$  & $\gamma_1$   \\
					\hline
		\end{tabular}
		\hfill
					\begin{tabular}[t]{cccc}
					\hline
					&Problem $P^v$&&\\
					\hline
					$o$&  $d$     & $l$ & $u$  \\
					\hline
					$O$& $(1\ 2\ 1\ 1\ 1)$     &$ \mu_2$  & $\mu_2$   \\				
					$O$& $(1\ 2\ 4\ 1\ 1)$     &$ 0$  & $\alpha_1$   \\				
					$O$& $(1\ 2\ 5\ 1\ 1)$     &$ 0$  & $\alpha_2$   \\				
					$O$& $(1\ 1\ 4\ 1\ 1)$     &$ 0$  & $\beta$   \\				
					$O$& $(1\ 2\ 4\ 0\ 1)$     &$ 0$  & $\beta$   \\				
					$(1\ 1\ 4\ 0\ 1)$ & $Z$    & $\gamma_2$  & $\gamma_2$   \\
					$(1\ 0\ 4\ 1\ 1)$ & $Z$    & $\gamma_2$  & $\gamma_2$   \\
					$(1\ 1\ 1\ 1\ 1)$ & $Z$    & $0$  & $\mu_1$   \\
					$(1\ 1\ 0\ 1\ 1)$ & $Z$    & $0$  & $\mu_2$   \\
					$(1\ 2\ 1\ 1\ 1)$ & $Z$    & $0$  & $\gamma_1$   \\
					$(1\ 3\ 1\ 1\ 1)$ & $Z$    & $0$  & $\gamma_2$   \\
					$(1\ 1\ 1\ 2\ 1)$ & $Z$    & $0$  & $\gamma_1$   \\
					$(1\ 1\ 1\ 3\ 1)$ & $Z$    & $0$  & $\gamma_2$   \\
									\hline
\end{tabular}		
\hfill
					\begin{tabular}[t]{cccc}
					\hline
					&Problem $P^{v+1}$&&\\
					\hline
					$o$&  $d$     & $l$ & $u$  \\
					\hline
					$O$& $(1\ 2\ 3\ 0\ 1)$     &$\mu_1$  & $\mu_1$   \\				
					$O$& $(1\ 2\ 4\ 0\ 1)$     &$ \beta$  & $\beta$   \\				
					$O$& $(1\ 2\ 3\ 0\ 2)$     &$ \beta$  & $\beta$   \\				
					$O$& $(1\ 2\ 3\ 2\ 1)$     &$ 0$  & $\alpha_1$   \\				
					$O$& $(1\ 2\ 3\ 3\ 1)$     &$ 0$  & $\alpha_2$   \\				
					$O$& $(1\ 2\ 2\ 2\ 1)$     &$ 0$  & $\beta$   \\				
					$O$& $(1\ 2\ 3\ 2\ 0)$     &$ 0$  & $\beta$   \\				
					$(1\ 1\ 1\ 2\ 1)$ & $Z$    & $\gamma_1$  & $\gamma_1$   \\
					$(1\ 1\ 1\ 3\ 1)$ & $Z$    & $\gamma_2$  & $\gamma_2$   \\
					$(1\ 1\ 2\ 2\ 0)$ & $Z$    & $\gamma_1$& $\gamma_1$   \\
					$(1\ 1\ 2\ 3\ 0)$ & $Z$    & $\gamma_2$  & $\gamma_2$   \\
					$(1\ 1\ 2\ 0\ 1)$ & $Z$    & $0$  & $\mu_1$   \\
					$(1\ 1\ 3\ 0\ 1)$ & $Z$    & $0$  & $\gamma_1$   \\
					$(1\ 1\ 4\ 0\ 1)$ & $Z$    & $0$  & $\gamma_2$   \\
					$(1\ 1\ 2\ 0\ 2)$ & $Z$    & $0$  & $\gamma_1$   \\
					$(1\ 1\ 2\ 0\ 3)$ & $Z$    & $0$  & $\gamma_2$   \\
					\hline
		\end{tabular}
		}}
	\end{center}
\end{table}

\begin{table}[!h]
	\begin{center}
				{\caption{\label{2solp1}Solutions of problems $P^{v-1}$, $P^v$ and $P^{v+1}$. Only nonzero flows from from $b\in T_2^v$ to  $a\in T_1^v$ are shown.}}
		{\scalebox{1}{
				\begin{tabular}[t]{ccc}
					\hline
					& $P^{v-1}$&\\
					\hline
					$o$&  $d$     & $f^{v-1}$  \\
					\hline
					$(1\ 0\ 3\ 1\ 1)$ & $(1\ 0\ 2\ 1\ 1)$     & $\mu_1$   \\
					$(1\ 0\ 3\ 1\ 1)$ & $(1\ 0\ 3\ 1\ 1)$     & $\mu_2-\mu_1$   \\
					$(1\ 4\ 3\ 1\ 1)$ & $(1\ 3\ 1\ 1\ 1)$     & $\gamma_2$   \\
					$(1\ 4\ 3\ 1\ 1)$ & $(0\ 3\ 2\ 1\ 1)$     & $\gamma_2$   \\
					\hline
		\end{tabular}
		\hspace{0.5cm}
					\begin{tabular}[t]{ccc}
					\hline
					&Problem $P^v$&\\
					\hline
					$o$&  $d$     & $f^v$  \\
					\hline
					$(1\ 2\ 1\ 1\ 1)$  & $(1\ 1\ 0\ 1\ 1)$     & $\mu_2$   \\
					$(1\ 2\ 5\ 1\ 1)$  & $(1\ 1\ 4\ 0\ 1)$     & $\gamma_2$   \\
					$(1\ 2\ 5\ 1\ 1)$  & $(1\ 0\ 4\ 1\ 1)$     & $\gamma_2$   \\
									\hline
\end{tabular}		
\hspace{0.5cm}
					\begin{tabular}[t]{ccc}
					\hline
					&Problem $P^{v+1}$&\\
					\hline
					$o$&  $d$     & $f^{v+1}$  \\
					\hline
					$(1\ 2\ 3\ 0\ 1)$ & $(1\ 1\ 2\ 0\ 1)$     & $\mu_1$   \\
					$(1\ 2\ 4\ 0\ 1)$ & $(1\ 1\ 3\ 0\ 1)$     & $2\gamma_1-\beta$   \\
					$(1\ 2\ 4\ 0\ 1)$ & $(1\ 1\ 4\ 0\ 1)$     & $2(\beta-\gamma_1)$   \\
					$(1\ 2\ 3\ 0\ 2)$ & $(1\ 1\ 3\ 0\ 1)$     & $\beta-\gamma_1$   \\
					$(1\ 2\ 3\ 0\ 2)$ & $(1\ 1\ 2\ 0\ 2)$     & $\gamma_1$   \\
					$(1\ 2\ 3\ 3\ 1)$ & $(1\ 1\ 1\ 3\ 1)$     & $\gamma_2$   \\
					$(1\ 2\ 3\ 3\ 1)$ & $(1\ 1\ 2\ 3\ 0)$     & $\gamma_2$   \\
					$(1\ 2\ 2\ 2\ 1)$ & $(1\ 1\ 1\ 2\ 1)$     & $\gamma_1$   \\
					$(1\ 2\ 3\ 2\ 0)$ & $(1\ 1\ 2\ 2\ 0)$     & $\gamma_1$   \\
					\hline
		\end{tabular}
		}}
	\end{center}
\end{table}

After obtaining solutions to the individual problems, we formulate problems $P^{v-1,x}$ and $P^{v,v+1}$ in accordance with Section \ref{secprob2}. The details of these problems and their respective solutions are presented in Tables \ref{2probpares} and \ref{2solpares}, respectively. As in problems $P^{v-1},P^v,P^{v+1}$,  Table \ref{2probpares} does not show the arcs between $T_2^{wz}$ and $T_1^{wz}$ and, in Table \ref{2solpares}, only the positive flows on those arcs are displayed.

\begin{table}[!h]
	\begin{center}
				{\caption{\label{2probpares}Lower ($l$) and upper ($u$) bounds for arcs $(o,d)$ in problems $P^{v-1,v}$ and $P^{v,v+1}$. Arcs between $T_2^{wz}$ and $T_1^{wz}$ are omitted.}}
		{\scalebox{1}{
				\begin{tabular}[t]{cccc}
					\hline
					& $P^{v-1,v}$&&\\
					\hline
					$o$&  $d$     & $l$ & $u$  \\
					\hline
					$O$& $(1\ 2\ 5\ 1\ 1)$     &$ 0$  & $\gamma_2$   \\				
					$O$& $(1\ 4\ 3\ 1\ 1)$     &$ 0$  & $\gamma_2$   \\				
					$(1\ 0\ 4\ 1\ 1)$ & $Z$    & $\gamma_2$  & $\gamma_2$   \\
					 $(1\ 3\ 1\ 1\ 1)$ & $Z$    & $\gamma_2$  & $\gamma_2$   \\
					\hline
		\end{tabular}
		\hspace{2cm}
					\begin{tabular}[t]{cccc}
					\hline
					&$P^{v,v+1}$&&\\
					\hline
					$o$&  $d$     & $l$ & $u$  \\
					\hline
					$O$& $(1\ 2\ 4\ 0\ 1)$     &$ \beta$  & $\beta$   \\				
					$O$& $(1\ 2\ 5\ 1\ 1)$     &$ 0$  & $\gamma_2$   \\				
					$O$& $(1\ 2\ 3\ 3\ 1)$     &$ 0$  & $\gamma_2$   \\				
					$O$& $(1\ 2\ 2\ 2\ 1)$     &$ 0$  & $\gamma_1$   \\				
					$(1\ 1\ 4\ 0\ 1)$  & $Z$    & $\gamma_2$  & $\gamma_2$   \\
					$(1\ 1\ 1\ 2\ 1)$  & $Z$    & $\gamma_1$  & $\gamma_1$   \\
					$(1\ 1\ 1\ 3\ 1)$  & $Z$    & $\gamma_2 $ & $\gamma_2$   \\
					$(1\ 1\ 3\ 0\ 1)$  & $Z$    & 0  & $2\gamma_1-\beta$   \\
									\hline
\end{tabular}		
		}}
	\end{center}
\end{table}

\begin{table}[!h]
	\begin{center}
				{\caption{\label{2solpares}Solutions of  problems $P^{v-1,v}$ and $P^{v,v+1}$. Only positive flow of arcs between $T_2^{wz}$ and $T_1^{wz}$ is shown}}
		{\scalebox{1}{
				\begin{tabular}[t]{ccc}
					\hline
					& $P^{v-1,v}$&\\
					\hline
					$o$&  $d$     & $f^{v-1,v}$  \\
					\hline
					$(1\ 2\ 5\ 1\ 1)$     & $(1\ 0\ 4\ 1\ 1)$     & $\gamma_2$   \\
					$(1\ 4\ 3\ 1\ 1)$     & $(1\ 3\ 1\ 1\ 1)$     & $\gamma_2$   \\
					\hline
		\end{tabular}
		\hspace{2cm}
					\begin{tabular}[t]{ccc}
					\hline
					&$P^{v,v+1}$&\\
					\hline
					$o$&  $d$     & $f^{v,v+1}$  \\
					\hline
					$(1\ 2\ 4\ 0\ 1)$  & $(1\ 1\ 4\ 0\ 1)$     &  $2(\beta-\gamma_1)$   \\
					$(1\ 2\ 4\ 0\ 1)$  & $(1\ 1\ 3\ 0\ 1)$     &  $2\gamma_1-\beta$   \\
					$(1\ 2\ 5\ 1\ 1)$  & $(1\ 1\ 4\ 0\ 1)$     &  $\gamma_2-2(\beta-\gamma_1)$   \\
					$(1\ 2\ 3\ 3\ 1)$  & $(1\ 1\ 1\ 3\ 1)$     &  $\gamma_2$   \\
					$(1\ 2\ 2\ 2\ 1)$  & $(1\ 1\ 1\ 2\ 1)$     &  $\gamma_1$   \\
									\hline
\end{tabular}		
		}}
	\end{center}
\end{table}

With all the flows  we can determine the rates of the coupled process, from configurations $\eta$ and $\xi$, involving site $v$. The resulting rates are provided in Table \ref{coupexp2}, where the last column shows the corresponding term out of \eqref{eq:gen3} to \eqref{eq:gen6}, which includes the change. 

\begin{table}[!h]
	\begin{center}
				{\caption{\label{coupexp2}Coupled changes involving site $v$.}}
		{\scalebox{1}{
				\begin{tabular}[t]{cccc}
					\hline
					$\eta_t$&  $\xi_t$     & Rate  & Term of the coupling \\
					\hline
	$(1\ 1\ 0\ 1\ 1)$     &				$(1\ 2\ 1\ 1\ 1)$ &$ \mu_2$ &\eqref{eq:gen3}  \\				
	$(1\ 0\ 3\ 1\ 1)$     &				$(1\ 0\ 3\ 1\ 1)$ &$ \mu_2-\mu_1$ &\eqref{eq:gen3}  \\				
	$(1\ 1\ 3\ 0\ 1)$     &				$(1\ 2\ 4\ 0\ 1)$ &$ 2\gamma_1-\beta$ &\eqref{eq:gen4}  \\				
	$(1\ 0\ 4\ 1\ 1)$     &				$(1\ 2\ 5\ 1\ 1)$ &$ \gamma_2$ &\eqref{eq:gen6}  \\				
	$(1\ 3\ 1\ 1\ 1)$     &				$(1\ 4\ 3\ 1\ 1)$ &$ \gamma_2$ &\eqref{eq:gen6}  \\				
	$(1\ 1\ 1\ 3\ 1)$     &				$(1\ 2\ 3\ 3\ 1)$ &$ \gamma_2$ &\eqref{eq:gen6}  \\				
	$(1\ 1\ 1\ 2\ 1)$     &				$(1\ 2\ 2\ 2\ 1)$ &$ \gamma_1$ &\eqref{eq:gen6}  \\				
	$(1\ 1\ 4\ 0\ 1)$     &				$(1\ 2\ 5\ 1\ 1)$ &$ \gamma_2-2(\beta-\gamma_1)$ &\eqref{eq:gen6}  \\				
	$(1\ 1\ 4\ 0\ 1)$     &				$(1\ 2\ 4\ 0\ 1)$ &$ 2(\beta-\gamma_1)$ &\eqref{eq:genx}  \\				
					\hline
\end{tabular}		
		}}
	\end{center}
\end{table}

\section*{Acknowledgement}
Financial support  from grants PIA AFB-170001, Fondecyt 1161319 (Chile) and 
PID2020-116873GB-I00, funded by MCIN$/$AEI$/$10.13039$/$501100011033 (Spain),
 is gratefully acknowledged. The authors are members of the research group Modelos Estoc\'{a}sticos E46$\_$23R of DGA (Arag\'on, Spain).

\bibliographystyle{imsart-number} 
\bibliography{SIGDDR2}       
\end{document}